\documentclass[12pt,oneside]{amsart}

\usepackage{hyperref}
\hypersetup{
    colorlinks=true,
    linkcolor=black,
    citecolor=magenta,
    filecolor=black,      
    urlcolor=blue,
    }

\usepackage{fullpage}
\usepackage{amssymb}
\usepackage{latexsym}
\usepackage{amsmath}
\usepackage{mathrsfs}
\usepackage[all]{xy}
\usepackage{enumerate}
\usepackage{enumitem}
\usepackage{amsthm}
\usepackage{psfrag}
\usepackage{ifthen}
\usepackage{mathdots}
\usepackage{caption}
\usepackage{subcaption}
\usepackage{standalone}
\usepackage{mathtools}
\usepackage{pgfplots}
\usepackage{comment}
\usepackage{xfrac}
\usepackage{todonotes}
\usepackage[utf8]{inputenc}
\usepackage[linesnumbered,ruled,vlined]{algorithm2e}
\usepackage[outline]{contour}
\usepackage{float}
\usepackage{tikz}
\usetikzlibrary{shapes.geometric}
\usepackage{graphicx}
\newcommand{\newbold}[1]{\bgroup\contourlength{0.01em}\contour{black}{#1}\egroup}

\newtheorem{theorem}{Theorem}[section]
\newtheorem{lemma}[theorem]{Lemma}
\newtheorem{metalemma}[theorem]{Meta Lemma}
\newtheorem{proposition}[theorem]{Proposition}
\newtheorem{corollary}[theorem]{Corollary}

\newtheorem*{theorem*}{Theorem}          % Unnumbered theorem

\theoremstyle{definition}
\newtheorem{definition}[theorem]{Definition}

\newtheorem{example}[theorem]{Example}

\usepackage{tabularx,environ}

\makeatletter
\newcommand{\problemtitle}[1]{\gdef\@problemtitle{#1}}
\newcommand{\probleminput}[1]{\gdef\@probleminput{#1}}
\newcommand{\problemquestion}[1]{\gdef\@problemquestion{#1}}
\NewEnviron{algproblem}{
  \problemtitle{}\probleminput{}\problemquestion{}% Default input is empty
  \BODY% Parse input
  \par\addvspace{.5\baselineskip}
  \noindent
  \begin{tabularx}{\textwidth}{@{\hspace{0em}} l X c}
    \hline\hline
    \multicolumn{2}{@{\hspace{0em}}l}{\@problemtitle} \\
    \textbf{In:} & \@probleminput \\% 
    \textbf{Question:} & \@problemquestion\\
    \hline\hline
  \end{tabularx}
  \par\addvspace{.5\baselineskip}
}
\makeatother

\newcommand{\bel}[1]{\begin{equation}\label{#1}}
\newcommand{\ee}{\end{equation}}

\newcommand{\LBA}{\left\{\begin{array}}
\newcommand{\EAR}{\end{array}\right.}

\def\ovu{{\overline{u}}}
\def\ovv{{\overline{v}}}

\def\CG{{\mathcal G}}

\def\CP{{\mathcal P}}

\def\NP{{\mathbf{NP}}}
\def\DP{{\mathbf{DP}}}

\def\blf{{\protect\newbold{$f$}}}
\def\blg{{\protect\newbold{$g$}}}
\def\blo{{\protect\newbold{$0$}}}

\def\one{{\mathbf{1}}}
\def\SD{{\,\triangle\,}}

\newcommand{\gpr}[2]{{\left\langle #1 \mid #2 \right\rangle}}
\newcommand{\rb}[1]{{\left( #1 \right)}}

\newcommand{\Set}[2]{\left\{\, #1 \;\middle|\; #2 \,\right\}}%set

\def\ME{{\mathbb{E}}}
\def\MN{{\mathbb{N}}}
\def\MZ{{\mathbb{Z}}}

\DeclareMathOperator{\DIV}{{DIV}}
\DeclareMathOperator{\DIVw}{{DIV^{wit}}}

\DeclareMathOperator{\size}{{size}}
\DeclareMathOperator{\supp}{{supp}}

\DeclareMathOperator{\ord}{{ord}}
\DeclareMathOperator{\num}{{num}}
\DeclareMathOperator{\den}{{den}}
\DeclareMathOperator{\pres}{{pres}}
\DeclareMathOperator{\PWP}{{PWP}}

\DeclareMathOperator{\Path}{{path}}
\DeclareMathOperator{\flip}{{flip}}

\title{One variable equations over the lamplighter group}
\author{Alexander Ushakov and Yankun Wang}
\address{Department of Mathematical Sciences, Stevens Institute of Technology, Hoboken NJ 07030}\email{aushakov,ywang7@stevens.edu}

\date{\today}

\begin{document}

\maketitle

\begin{abstract}
We prove that one variable equations in the lamplighter group $\MZ_2\wr \MZ$
are decidable and describe an algorithm for solving such equations.
The algorithm has super-exponential time complexity in the worst case.
We also show that, for most equations, decidability can be determined 
in nearly quadratic time; that is, the problem admits a nearly quadratic-time 
solution in the generic case.

\end{abstract}

\section{Introduction}

Let $F = F(Z)$ denote the free group on countably many generators 
$Z = \{z_i\}_{i=1}^\infty$. For a group $G$, an \emph{equation over $G$ with variables in $Z$} is an equation of the form $w = 1$, where $w \in F*G$. If $w = z_{i_1}g_1\cdots z_{i_k}g_k$, with $z_{i_j}\in Z$ and $g_j\in G$, then we refer to $\{z_{i_1},\ldots,z_{i_k}\}$ as the set of \emph{variables} and to $\{g_1,\ldots,g_k\}$ as the set of \emph{constants} (or \emph{coefficients}) of $w$.

A \emph{solution} to an equation $w(z_1,\ldots,z_k)=1$ over $G$ is a homomorphism $$\varphi\colon F*G\rightarrow G $$
such that $\varphi|_G = 1_G$ and $w \in \ker \varphi$. If $\varphi$ is a solution of $w =1$ and $g_i = \varphi(z_i)$, then we often say that $g_1,\ldots,g_k$ is a solution of $w =1$.

We assume that $G$ comes equipped with a fixed generating set $X$ 
and elements of $G$ are given as products of elements of $X$ and their inverses. 
This naturally defines the length (or size) of the equation $w =1$ as the length 
of its left-hand side $w $.

The \emph{Diophantine problem}, $\DP_C(G)$, over a group $G$ 
for a class of equations $C$ 
is an algorithmic question to decide whether a given equation $w =1$ in $C$ 
has a solution in $G$.
It can be regarded as a group-theoretic analogue of the satisfiability problem.
More generally, one can study the Diophantine problem for systems of
equations.
In this paper, we focus on the class of one-variable equations over 
the lamplighter group $\MZ_2\wr \MZ$.

We say that $w =1$ is a \emph{one-variable equation} if $w$ involves a single
variable, for convenience we assume it is $x$, one or more times
as $x$ or as $x^{-1}$.

\begin{algproblem}
\problemtitle{\textsc{Diophantine problem for one-variable equations over $G$} $(\DP_1(G))$.}
\probleminput{A group word $w =w(X,x) \in F(x)\ast G$, where $X$ is a generating set for $G$.}
\problemquestion{Is there a group word $w(X)$ such that $w(X,w)=1$ in $G$?}
\end{algproblem}

\subsection{Equations in the lamplighter and similar groups}

Computational properties of equations in the lamplighter group, 
in wreath products, and more generally 
in finitely generated metabelian groups have attracted 
considerable attention in recent years.

A. Miasnikov and N. Romanovsky \cite{Myasnikov-Romanovskii:2012} proved
that the Diophantine problem for systems of equations over $\MZ\wr \MZ$ 
is undecidable. 
R.~Dong \cite{Dong:2025} proved the same result using a different argument.
It is not known if the problem for single equations
over $\MZ\wr \MZ$ is decidable or not.
I.~Lysenok and A.~Ushakov proved that the Diophantine problem 
over free metabelian groups is $\NP$-complete for 
spherical equations \cite{Lysenok-Ushakov:2015}, and more generally, 
for orientable equations \cite{Lysenok-Ushakov:2021}.
The Diophantine problem 
for orientable quadratic equations over wreath products 
of finitely generated abelian groups is $\NP$-complete \cite{Ushakov-Weiers:2025}.
Also, the Diophantine problem for quadratic equations is $\NP$-complete
over metabelian Baumslag--Solitar groups \cite{Mandel-Ushakov:2023b}, and
the lamplighter group \cite{Ushakov-Weiers:2023}.
Existence of a finitely generated abelian-by-cyclic group with undecidable spherical equations was established in \cite{Dong:2025}.
It was claimed in \cite{Kharlampovich-Lopez-Miasnikov:2020} 
that the Diophantine problem for systems of equations over $\MZ_2 \wr \MZ$ 
is decidable. However, a gap was later discovered in the proof, 
and the status of the problem
for systems over $\MZ_2 \wr \MZ$ remains open.
A brief survey on the solvability of equations in wreath products of groups 
can be found in \cite{Bartholdi-Dong-Pernak-Waechter:2024}.

Applications of quadratic equations (particularly spherical equations) 
to cryptography have also been explored. In particular, 
\cite{Ushakov:2024} establishes connections
between computational group theory and lattice-based cryptography 
via spherical equations over a certain class of finite metabelian groups.

\subsection{Our contribution}

As mentioned above, it remains an open question whether 
the Diophantine problem for arbitrary single equations over the lamplighter group 
is decidable. However, decidability is established for certain classes of equations, 
such as quadratic ones. In this paper, we show that one-variable equations form
another such class: they are decidable, with worst-case complexity bounded by 
exponential time. Moreover, for a “typical” one-variable equation, 
existence of a solution can be decided in nearly quadratic time.

\subsection{Study of one-variable equations}

Obviously, one-variable equations can be solved efficiently in f.g. abelian groups.

Lyndon \cite{Lyndon:1960(2)} was the first to study 
one-variable equations over free groups. 
He characterized solution sets in terms of parametric words. 
The parametric words involved were simplified by Lorents 
\cite{Lorents:1963,Lorents:1968} and Appel \cite{Appel:1968}. 
However, Lorents announced his results without proof, 
and Appel’s published proof has a gap (see \cite{Chiswell-Remeslennikov:2000}). 
A complete proof has been provided by Chiswell and Remeslennikov. 
Solution sets of one-variable equations over free groups were studied via 
context-free languages in \cite{Gilman-Myasnikov:2004}, leading to a 
(high degree) polynomial-time algorithm of \cite{Bormotov-Gilman-Miasnikov:2008}.
R. Ferens and A. Je\.{z} describe a cubic time algorithm for this
problem in \cite{Ferens-Jez:2021}.

One-variable equations over nilpotent groups were studied by N. Repin 
in \cite{Repin:1984}.
He proved that for any \(c\ge 10^{20}\), there is no algorithm 
that can recognize the solvability of equations with one indeterminate 
in free nilpotent groups of class \(c\).
Yet, such an algorithm does exist 
for finitely generated nilpotent groups of class $2$,
also see \cite{Levine2022_virtually_class2_nilpotent} and \cite{DuchinLiangShapiro2015_equations_nilpotent_groups}.
Repin also constructs a finitely generated nilpotent group of class $3$
for which no algorithm exists to recognize the solvability of equations 
in one unknown. 

% ==============================

\subsection{Model of computation and data representation}
\label{se:model-computations}

We use a \emph{Random-Access Machine (RAM)} as a model of computation.
The RAM consists of an (in principle) unbounded collection of registers or memory cells, 
each capable of storing an integer (or, in certain variants, real numbers).
In general,
it is equipped with a fixed instruction set,
including arithmetic operations 
(addition, subtraction, multiplication, division), 
data movement (load and store), 
indirect addressing,
conditional and unconditional jumps, 
and a halt instruction. 
Each instruction is usually counted as 
a single computation step, 
which allows a natural definition of time complexity,
in the sense of combinational group theory.

The \emph{representation of data} within the RAM model plays a critical role in the analysis of algorithms. 
In particular,
integers are stored in registers, 
and Boolean values are represented by 
distinguished integers (e.g., $0$ for \texttt{false}, $1$ for \texttt{true}). 
Furthermore, the complex data types 
are encoded as sequences of integers occupying contiguous registers,
characters and strings may be stored via standard encodings such as ASCII or Unicode, 
arrays are represented by contiguous memory blocks with constant-time index access, 
and graphs or matrices can be encoded by adjacency structures or multi-dimensional arrays.
Such representations allow abstract data structures to be simulated within the RAM model 
while preserving a clear notion of storage cost and access time.

Different RAM variants change which operations are treated as unit-cost and thus alter the asymptotic analysis.
There are two common RAMs, the \emph{Word RAM} and the \emph{transdichotomous} models.
The \emph{Word RAM} assumes a word length \(w\) (possibly depending on the problem size \(n\)) so that
integers up to size about \(2^{w}\) fit in one word and basic arithmetic/bitwise operations
and word reads/writes are constant time.
Moreover, the transdichotomous viewpoint (word size tied to problem size)
is standard when deriving bounds for integer sorting, priority queues, and related structures.

\subsection{Outline}

In Section~\ref{se:Laurent} we review the definition of Laurent polynomials and examine the computational complexity of their basic arithmetic operations.
In Section~\ref{se:lamplighter} we review the definition of the lamplighter group.
In Section~\ref{se:equations-to-divisibility} 
we introduce \emph{$\delta$-parametric polynomials}
and prove that the Diohpantine problem $\DP_1(L_2)$
for one-variable equations over $L_2$ reduces in polynomial-time
to divisibility problem for $\delta$-parametric polynomials.
In Section~\ref{se:num-den} we discuss properties of the parametric 
polynomials that appear as we reduce equations to divisibility.
In Section~\ref{se:division-automaton} we introduce an
auxiliary tool called the \emph{division-by-$f$ automaton}.
In Section~\ref{se:divisibility-decidable}
we prove that the divisibility problem for $\delta$-parametric polynomials
is decidable.

\section{Preliminaries: Laurent polynomials}
\label{se:Laurent}

%\subsection{Laurent polynomials}

Let $R$ be a commutative ring with unity.
A \emph{Laurent polynomial} over $R$ is an expression of the form
$$
\sum_{i=-\infty}^\infty a_i z^i\ \ \ (a_i\in R),
$$
where all but finitely many \emph{coefficients} $a_i$ 
are trivial. The sum of two polynomials is defined by
$$
\sum_{i=-\infty}^\infty a_i z^i +
\sum_{i=-\infty}^\infty b_i z^i = 
\sum_{i=-\infty}^\infty (a_i + b_i) z^i
$$
and the product of two polynomials is defined by
$$
\rb{\sum_{i=-\infty}^\infty a_i z^i} \cdot
\rb{\sum_{i=-\infty}^\infty b_i z^i} = 
\sum_{i=-\infty}^\infty c_i z^i,
\mbox{ where } c_i=\sum_{j=-\infty}^\infty a_j b_{i-j}.
$$
The set of all Laurent polynomials, denoted $R[z^\pm]$,
equipped with $+$ and $\cdot$ defined above is a ring, see \cite{CoxLittleOSheaIVA}.
Throughout the paper the ring of coefficients $R$ is
the finite field $\MZ_2=\{0,1\}$ of order $2$.

We say that a Laurent polynomial is \emph{trivial} if all its coefficients are zero.
For a nontrivial $f(z) = \sum_{i=-\infty}^\infty a_i z^i$ define
\begin{itemize}
\item 
$\deg(f) = \max \Set{i}{a_i\ne 0}$ called the \emph{degree} of $f$,
\item 
$\ord(f) = \min \Set{i}{a_i\ne 0}$ called the \emph{order} of $f$.
\end{itemize}
The monomial $a_{\deg(f)} z^{\deg(f)}$ in $f$
is called the \emph{leading monomial}
and $a_{\ord(f)} z^{\ord(f)}$ is called the \emph{trailing monomial}.
For the trivial polynomial $f=0$,
$\ord(f)$ and $\deg(f)$ are not defined
and we denote this by writing $\deg(f)=\varnothing$ and $\ord(f)=\varnothing$.
Define 
$$
\size(f) =
\begin{cases}
\max(|\deg(f)|,|\ord(f)|) & \mbox{if } f\ne 0,\\
0 & \mbox{if } f=0.\\
\end{cases}
$$

\subsection{Division in $\MZ_2[z^\pm]$}

To divide $f(z)$ by $g(z)\ne 0$ means to find $q(z),r(z) \in\MZ_2[z^\pm]$
satisfying $f(z) = g(z)\cdot q(z) + r(z)$, where
\begin{equation}\label{eq:r}
r=0\ \ \mbox{ or }\ \ 
\ord(f)\le \ord(r)\le \deg(r) <
\ord(f)+[\deg(g)-\ord(g)].
\end{equation}
If $r=0$, then we say that $g$ \emph{divides} $f$ and use notation $g\mid f$.
Call $q(z)$ and $r(z)$ the \emph{quotient} and \emph{remainder} of division resp.
Denote by $g\ \%\ f$ the remainder of division of $f$ by $g$.
There is a close relationship between division in $\MZ_2[z^\pm]$
and in $\MZ_2[z]$. Indeed, for any $f,g,r\in\MZ_2[z^\pm]$, where $g\ne 0$, we have
\begin{align*}
r = f\ \%\ g\ \  \Leftrightarrow\ \ & 
f=gq+r \mbox{ for some }q\in \MZ_2[z^\pm] \mbox{ satisfying \eqref{eq:r}}\\
\ \  \Leftrightarrow\ \ & 
\underbrace{f\cdot z^{-\ord(f)}}_{F\in\MZ_2[z]} = 
\underbrace{g \cdot z^{-\ord(g)}}_{G\in\MZ_2[z]}
\cdot 
\underbrace{q \cdot z^{-\ord(q)}}_{Q\in\MZ_2[z]} + 
\underbrace{r\cdot z^{-\ord(f)}}_{R}\\
& \scalebox{0.8}{(satisfying $R=0$ or $0\le \ord(R),\ \ \deg(R)<\deg(G)$)} \\
\ \  \Leftrightarrow\ \ & 
R = F\ \%\ G \mbox{ in } \MZ_2[z].
\end{align*}
In particular, the following holds.

\begin{lemma}\label{Lemma:equivalence of Laurent division}
$g(z)$ divides $f(z)$ in $\MZ_2[z^\pm]$ $\ \ \Leftrightarrow\ \ $
$G(z)$ divides $F(z)$ in $\MZ_2[z]$.
\end{lemma}

\subsection{Representation for Laurent polynomials}

We represent a nontrivial polynomial $f = \sum_{i=-\infty}^\infty a_i z^i$ 
by the pair $(\deg(f),w_f)$, where $\deg(f)$ is given in unary and
$w_f\in\{0,1\}^\ast$ is a bit string defined by
$$
w_f = 
\begin{cases}
a_{\ord(f)}\dots a_{\deg(f)} & \mbox{ if } f\ne 0,\\
\varepsilon  & \mbox{ if } f =  0,\\
\end{cases}
$$
and the trivial polynomial $f=0$ by the pair $(0,w_0)$.
Notice that so-defined value $\size(f)$ properly reflects 
the size of the representation for $f$ --
the number of bits required for $(\deg(f),w_f)$
is $\Theta(\size(f))$.
Notice that $w_f$ defined above starts and ends with $1$ if $w_f$ nonempty.
We call such representation \emph{strict}.

In some cases it is convenient to 
use a \emph{non-strict} representation for a polynomial $f$, defined 
by a pair $(D,w)$, where $\deg(f)\le D$, $w=a_d\dots a_D$, and
$d\le \ord(f)$. A non-strict representation allows $w$ to start and end with $0$.
We use notation $\pres(f)=(D,w)$ if $(D,w)$ defines the polynomial $f$.

\subsection{Complexity of operations in $\MZ_2[z^\pm]$}
\label{se:complexity-operations}

Let $R$ be any commutative ring with unity.
Addition, subtraction, multiplication, and division
for polynomials from $R[z]$ can be performed in linear or nearly linear time.
In more detail, the following holds.

\begin{theorem}[{{\cite[Section 2.2]{GathenGerhard2003}}}]
\label{th:plus}
Polynomials of degree less than $n$  can be added using at most $O(n)$
arithmetic operations in $R$.   
\end{theorem}

\begin{theorem}[{{\cite[Theorem 8.23]{GathenGerhard2003}}}]
\label{th:times}
Polynomials of degree less than $n$ can be multiplied using at most 
$(18+72 \log_3 2)n\log n \log\log n + O(n\log n)$
or $63.43n\log n \log\log n + O(n\log n)$
arithmetic operations in $R$.
\end{theorem}

Denote by $M(n)$ the number of arithmetic operations in $R$
required to multiply two arbitrary polynomials in $R[x]$ 
of degree less than $n$.
By Theorem \ref{th:times}, $M(n)\in O(n\log n \log \log n )$.

\begin{theorem}[{{\cite[Section 9.11]{GathenGerhard2003}}}]
\label{th:div}
Division with remainder of a polynomial $f\in R[x]$
by a monic polynomial $g\in R[x]$
of degree less than $n$ can be done using 
$O(M(n))$ arithmetic operations in $R$.
\end{theorem}

Addition, subtraction, multiplication, and division of Laurent polynomials 
in $\MZ_2[z,z^{-1}]$ can be reduced to the corresponding operations 
on ordinary polynomials in $\MZ_2[z]$.
Indeed, consider $f(z),g(z)\in \MZ_2[z^{\pm}]$. Define 
$$
o_f=\ord(f),\ o_g=\ord(g),\ \mbox{ and } m=-\min(o_f,o_g),
$$
and notice that in this notation
\begin{align}
\label{eq:plus}
f(z) \pm g(z) &= (f(z)z^{-m} \pm g(z)z^{-m})z^m\\
\label{eq:times}
f(z) \cdot g(z) &= (f(z)z^{-o_f} \cdot g(z)z^{-o_g})z^{o_f+o_g}\\
\label{eq:div}
g(z) \% f(z) &= (g(z)z^{-o_g}\ \%\ f(z)z^{-o_f})z^{o_f}.
\end{align}

\begin{proposition}\label{pr:Laurent-complexity}
Let $f(z),g(z)\in \MZ_2[z^{\pm}]$ and 
$
n = \max\{\size(f),\size(g)\}
$.
\begin{itemize}
\item[(a)]
There is an algorithm that computes $f(z)\pm g(z)$ in $O(n)$ time.
\item[(b)]
There is an algorithm that computes $f(z)\cdot g(z)$ in $O(n\log n\log \log n)$ time.
\item[(c)]
There is an algorithm that computes $g(z)\ \%\ f(g)$ in $O(n\log n\log \log n)$ time.
\item[(d)]
There is an algorithm that decides 
whether $g(z)\ne 0$ divides $f(z)$ in $O(n\log n\log \log n)$ time.
\end{itemize}
\end{proposition}

\begin{proof}
Theorems \ref{th:plus}, \ref{th:times}, and \ref{th:div} allow to compute
$$
f(z)z^{-m} \pm g(z)z^{-m},\ 
f(z)z^{-o_f} \cdot g(z)z^{-o_g},\ \mbox{ and }
g(z)z^{-o_g}\ \%\ f(z)z^{-o_f},
$$
in $O(n)$, $O(n\log n\log \log n)$, and
$O(n\log n\log \log n)$ time resp.
Multiplying the result by $z^m, z^{o_f+o_g}$, and $z^{o_f}$
we compute
\eqref{eq:plus}, \eqref{eq:times}, and \eqref{eq:div}.
\end{proof}

In Section \ref{se:divisibility-decidable}, we also employ 
the \emph{synthetic division} algorithm (see \cite[Section 2.5]{GathenGerhard2003}), 
whose complexity can be estimated as $O(n^2)$ arithmetic operations in $R$.

% The division of Laurent polynomials can be also viewed as
% a finite steps elimination process of monomials,
% which halts when the size of reminder polynomial is less than the size of divisor polynomial.
% Precisely speaking,
% for two given Laurent polynomials $(\deg(f), w_{f})$ and $(\deg(g), w_{g})$, the process of division of $g$ by $f$ can be viewed as follows:

% \begin{itemize}
% \item $w_{g} \leftarrow w_g - w_f \cdot z^{\deg(g)- \deg(f)}$;
% \item check the degree of $w_g$, if $\deg(g) < \deg(f)$ and $w_g \neq 0$, then $g$ is not divisible by $f$, if not, repeat above process.
% \end{itemize}

% The process will terminate when the degree of dividend polynomial is strictly less than the degree of divisor polynomial or dividend polynomial is trivial, and the complexity depends on how many operations it did in the first step.

\section{Preliminaries: the lamplighter group}
\label{se:lamplighter}

Define the support of a function $f\colon \MZ\to \MZ_2$ as a set
$$
\supp(f) =\Set{i\in \MZ}{f(i)\ne 0}.
$$
Define a set 
$$
\MZ_2^\MZ  = \Set{f\colon \MZ\to \MZ_2}{|\supp(f)|<\infty}
$$
and a binary operation $+$ on $\MZ_2^\MZ$ which for
$f,g\in \MZ_2^\MZ$ produces $f+g\in \MZ_2^\MZ$ defined by
$$
(f+g)(z) = f(z)+g(z) \ \mbox{ for } \ z\in \MZ.
$$
For a nontrivial $f$ it will be convenient to define
$$
m(f)=\min_{z\in \supp(f)} x
\ \ \mbox{ and }\ \ 
M(f)=\max_{z\in \supp(f)} x.
$$
For $f\in \MZ_2^\MZ$ and $b\in \MZ$, define $f^b\in \MZ_2^\MZ$ by
$$
f^b(z)=f(z+b) \
\mbox{ for } \ z\in \MZ,
$$
which is a right $\MZ$-action on $\MZ_2^\MZ$ because $f^0=f$ and 
$f^{(b_1+b_2)}(z)=f(z+b_1+b_2)=(f^{b_1})^{b_2}(z)$.
Hence we can consider a semi-direct product $\MZ \ltimes \MZ_2^\MZ$
equipped with the operation
\begin{equation}\label{eq:semi-product}
(\delta_1,f_1)(\delta_2,f_2)=
(\delta_1+\delta_2,f_1^{\delta_2}+f_2).
\end{equation}
The group $\MZ \ltimes \MZ_2^\MZ$ is called the
\emph{restricted wreath product} of $\MZ_2$ and $\MZ$ and is denoted by
$\MZ_2\wr\MZ$. The group $\MZ_2\wr\MZ$ is also well-known as the \textit{lamplighter group}(\cite{Ushakov-Weiers:2023}), 
as it can be viewed as an infinite set of lamps (each lamp indexed by an element 
of $\MZ$), with each lamp either on or off, and a lamplighter positioned at some lamp.
Given some element $(\delta,f)\in \MZ_2\wr\MZ$, $f\in\MZ_2^\MZ$ represents the 
configuration of illuminated lamps and $\delta\in\MZ$ represents the position 
of the lamplighter. Henceforth, we denote $\MZ_2\wr\MZ$ as $L_2$.

There is a natural (abelian group) isomorphism between $\MZ_2^\MZ$ and the 
ring  $\MZ_2[z^\pm]$ of Laurent polynomials
with coefficients in $\MZ_2$ that maps 
$f\in \MZ_2^\MZ$ to $\sum_{i=-\infty}^\infty f(i)z^i$.
Hence, the group $L_2$ can be viewed 
as $\MZ\ltimes \MZ_2[z^\pm]$ with the $\MZ$-action on $\MZ_2[z^\pm]$
defined by
$$
f^\delta = f\cdot z^{-\delta}.
$$

% \subsection{Presentation of the lamplighter group $L_2$}

% The \textit{lamplighter group} $L_2$ can be viewed as a dynamical system as well({\cite{Lamplightergroup:2018}}), and there are two independently tasks for the lamplighter such that for any reconfiguration, the lamplighter performs only finitely many tasks.
% These tasks can be interpreted as functions, or more precisely, generators which are defined by $a$ and $t$({\cite{Ushakov-Weiers:2023}}).

It is easy to see that $L_2$ is a two-generated group.
In particular, it can be generated by the elements
$a=(0,\mathbf{1}_0)$ and $t = (1,0)$, where $\one_0 \in \MZ_2^\MZ$ is defined by
$$
\one_0(x)=
\begin{cases}
1& \mbox{ if } x=0\\
0& \mbox{ if } x\ne0\\
\end{cases}
$$
and $0$ in the second component of $t$ is the function identical to $0$.
G. Baumslag in \cite{Baumslag:1961} proved that $L_2$ is not finitely presented,
but can be defined using the following infinite presentation:
$$
L_2 \simeq \gpr{a,t}{a^2, [a^{t^i}, a^{t^j}](i,j \in \MZ)}.
$$

\begin{lemma}\label{le:word-to-pair}
It takes $O(|w|)$ time to compute the pair $(\delta,f) \in \MZ\ltimes \MZ_2[z^\pm]$
for a given group-word $w=w(a,t)$.
\end{lemma}

\begin{proof}
One can process $w$ letter-by-letter from the right to the left using the formula
\eqref{eq:semi-product} with $(\delta_1,f_1) = (0,\mathbf{1}_0)^\pm$ or $(1,0)^\pm$
that can be handled using formulas
$$
(0,\mathbf{1}_0)^\pm \cdot (\delta_2,f_2) = (\delta_2,f_2+z^{\delta_2})
\ \ \mbox{ and }\ \ 
(1,0)^\pm \cdot (\delta_2,f_2) = (\delta_2\pm 1,f_2).
$$
This can be done in $O(|w|)$ time.
\end{proof}

Finally, we formulate the Diophantine problem for one-variable equations over $L_2$.

\begin{algproblem}
\problemtitle{\textsc{Diophantine problem for one-variable equations over $L_2$}, $(\DP_1(L_2))$.}
\probleminput{A group word $w=w(a,t,x)\in F(a,t,x)$.}
\problemquestion{Is there $s=s(a,t)\in F(a,t)$ satisfying $w(a,t,s(a,t))=1$ 
in $L_2$?}
\end{algproblem}

For a word $w=w(a,t,x)\in F(a,t,x)$ define
\begin{itemize}
\item 
$\sigma_a(w) = $ the sum of exponents for $a$ in $w$,
\item 
$\sigma_t(w) = $ the sum of exponents for $t$ in $w$,
\item 
$\sigma_x(w) = $ the sum of exponents for $x$ in $w$.
\end{itemize}
$\sigma_a$, $\sigma_t$, and $\sigma_x$
are computable in linear time $O(|w|)$ for a given $w=w(a,t,x)$.

\subsection{Useful formulae}

For any $(\delta,f),(\delta_1,f_1) \in \MZ_2\wr \MZ$
the following holds:
\begin{align*}
(\delta,f)^{-1} &= (-\delta,-f^{-\delta})
&& \mbox{(inverse)}\\
(\delta,f)^{-1} (\delta_1,f_1) (\delta,f) &= 
(\delta_1,(1-z^{-\delta_1})f+z^{-\delta}f_1)
&& \mbox{(conjugation)}\\
[(\delta, f), (\delta_1, f_1)] &= (0, -f - f_1^ {\delta} + f^{\delta_1} + f_1)
&& \mbox{(commutator)}\\
&= (0, -f(1-z^{-\delta_1}) + f_1(1-z^{-\delta}))\\
&= (0, f(1-z^{-\delta_1}) + f_1(1-z^{-\delta})).
\end{align*}

\section{One-variable equations: translation to parametric polynomials}
\label{se:equations-to-divisibility}

In this section we introduce \emph{$\delta$-parametric polynomials}
and prove that the Diohpantine problem $\DP_1(L_2)$
for one-variable equations over $L_2$ reduces in polynomial-time
to divisibility problem for $\delta$-parametric polynomials.

\begin{lemma}
It requires linear time to translate a word $w(a,t,x)$ into the form 
\begin{equation}\label{eq:one-var-eq}
c_0 x^{\varepsilon_1} c_1
\dots 
c_{k-1} x^{\varepsilon_k} c_k = 1,
\end{equation}
where $c_i=(\delta_i,f_i)\in \MZ \ltimes \MZ_2^\MZ$ are constants
and $\varepsilon_i=\pm 1$.
\end{lemma}

\begin{proof}
It is straightforward to express $w(a,t,x)$ as an alternating sequence 
$$
w_0(a,t) x^{\varepsilon_1} w_1(a,t)
\dots 
w_{k-1}(a,t) x^{\varepsilon_k} w_k(a,t)
$$
and then to apply Lemma \ref{le:word-to-pair} to each individual $w_i$.
\end{proof}

Hence, we may assume that the equation is given in the form \eqref{eq:one-var-eq}.
For $i=0,\dots,k$ define
\begin{align*}
\arraycolsep=3pt
\begin{array}{rcl}
\alpha_i
&=&\delta_{i+1}+\dots+\delta_{k}\\
\beta_i
&=&\varepsilon_{i+1}+\dots+\varepsilon_{k}
\end{array}
&&
\gamma_i =
\begin{cases}
\beta_i &\mbox{if } \varepsilon_i=1,\\
\beta_i-1 &\mbox{if } \varepsilon_i=-1.\\
\end{cases}
\end{align*}
With this notation, substituting $x=(\delta,f)$ into the left-hand 
side of \eqref{eq:one-var-eq} yields
\begin{align}\nonumber
(\delta_0,f_0)
(\delta,f)^{\varepsilon_1}
\dots
(\delta,f)^{\varepsilon_k}
(\delta_k,f_k) 
&= 
\Bigl(\sum_{i=0}^k \delta_i + \delta\sum_{i=1}^k \varepsilon_i,
\ \ 
\sum_{i=0}^k f_i^{\alpha_i+\beta_i\delta}+
\sum_{i=1}^k \varepsilon_i f^{\alpha_{i-1}+\gamma_i\delta}
\Bigr)\\
\nonumber
&= 
\Bigl(\sum_{i=0}^k \delta_i + \delta\sum_{i=1}^k \varepsilon_i,
\ \ 
\sum_{i=0}^k f_i z^{-\alpha_i-\beta_i\delta}+
f
\sum_{i=1}^k \varepsilon_i z^{-\alpha_{i-1}-\gamma_i\delta}
\Bigr)\\
\label{eq:rhs}
&= 
\Bigl(\underbrace{\sum_{i=0}^k \delta_i}_{t_w}
+ \delta\underbrace{\sum_{i=1}^k \varepsilon_i}_{x_w},
\ \ 
\underbrace{\sum_{i=0}^k f_i z^{-\alpha_i-\beta_i\delta}}_{\num(w)}+
f
\underbrace{\sum_{i=1}^k z^{-\alpha_{i-1}-\gamma_i\delta}}_{\den(w)}
\Bigr).
\end{align}
In the last formula the coefficients
$\varepsilon_i=\pm 1$ are omitted from the term for $\den$ 
because $-1=1$ in $\MZ_2$.
Denote by $t_w,x_w,\num(w),\den(w)$
the four components of the expression \eqref{eq:rhs}
\begin{align*}
t_w&=\sum_{i=0}^k \delta_i, &
x_w&=\sum_{i=0}^k \varepsilon_i, &
\num(w)&=\sum_{i=0}^k f_iz^{-\alpha_i-\beta_i\delta},&
\den(w)&=\sum_{i=1}^k z^{-\alpha_{i-1}-\gamma_i\delta}.
\end{align*}

\begin{lemma}
$\sigma_x(w) = x_w$ and $\sigma_t(w) = t_w$ for every $w=w(a,t,x)$.
\end{lemma}

\begin{proof}
The expressions $\sum_{i=1}^k \varepsilon_i$ and $\sum_{i=0}^k \delta_i$ 
count the sum of powers of $x/t$ in $w(a,t,x)$, i.e., it compute
$\sigma_x(w)$ and $\sigma_t(w)$ resp.
\end{proof}

\begin{lemma}\label{le:C1-C2}
$x=(\delta,f)$ satisfies  \eqref{eq:one-var-eq} 
$\ \Leftrightarrow\ $
the following holds for $\delta$ and $f$:
\begin{itemize}
\item[(C1)]
$\sum_{i=0}^k \delta_i + \delta\sum_{i=1}^k \varepsilon_i = 0$,
\item[(C2)]
$\sum_{i=0}^k f_i z^{-\alpha_i-\beta_i\delta}+
f\sum_{i=1}^k z^{-\alpha_{i-1}-\gamma_i\delta} = 0$.
\end{itemize}
\end{lemma}

\begin{proof}
Indeed, $x=(\delta,f)$ satisfies \eqref{eq:one-var-eq} if and only if
\eqref{eq:rhs} evaluates to $(0,0)$,
which yields the conditions (C1) and (C2).
\end{proof}

Condition (C1) defines two cases for the equation $w(a,t,x)=1$:
$x_w \ne 0$ and $x_w = 0$.
In the first case, when $x_w \ne 0$, 
condition (C1) uniquely defines the value of $\delta$
\begin{equation}
\label{eq:delta}
\delta =-\frac{t_w}{x_w},
\end{equation}
and (C2) uniquely defines the polynomial $f$ 
\begin{equation}
\label{eq:f}
f= -\frac{\num(w)}{\den(w)}
=\frac{\num(w)}{\den(w)}.
\end{equation}
In this case the Diophantine problem for $w(a,t,x)=1$
is clearly  decidable in polynomial time,
see Proposition \ref{Prop:algorithm-equations} for more detail.

When $\sum_{i=1}^k \varepsilon_i = 0$, 
condition (C1), when consistent, does not define $\delta$
and $\delta$ can attain any integer value.
Thus, in this case, to decide if $w(a,t,x)=1$ has a solution
one needs to check if $\den_{\delta}(w)$ divides $\num_{\delta}(w)$
for some $\delta\in\MZ$, which is an instance of divisibility problem defined
below.
The expressions $\num_\delta,\den_\delta$ are Laurent polynomials 
in which the value of $\delta$ is undefined.
We call such expressions $\delta$-parametric polynomials.

\subsection{Parametric polynomials}

A \emph{$\delta$-parametric polynomial} is an expression of the form
\begin{equation}\label{eq:param-poly}
\sum_{i=s}^t f_i(z) z^{i\delta},
\end{equation}
where $f_s(z),\dots,f_t(z) \in \MZ_2[z^\pm]$ are fixed Laurent polynomials
and $\delta$ is a \emph{parameter} (a variable) that takes values in $\MZ$.
We use bold letters, e.g., $\blf$ of $\blg$, to denote $\delta$-parametric polynomials.
To \emph{instantiate} a $\delta$-parametric polynomial means to
replace the variable $\delta$ with a particular integer value.
The result of an instantiation of $\blf$ with a value $\delta\in\MZ$
is a Laurent polynomial denoted by $\blf_\delta$.

\begin{algproblem}
 \problemtitle{\textsc{Divisibility for $\delta$-parametric polynomials} $(\DIV(\blf,\blg))$.}
  \probleminput{$\delta$-parametric polynomials $\blf$ and $\blg$.}
  \problemquestion{Is there $\delta\in\MZ$ such that $\blf_\delta\mid \blg_\delta$?}
\end{algproblem}

$\delta\in\MZ$ satisfying $\blf_\delta\mid \blg_\delta$ is called
a \emph{witness} for the instance $(\blf,\blg)$ of $\DIV$.
$\DIV$ naturally splits in two cases: $\delta\ge 0$ and $\delta<0$.
We claim that both cases can be treated the same way as instances 
of the following problem.

\begin{algproblem}
 \problemtitle{\textsc{Positive divisibility for $\delta$-parametric polynomials} $(\DIV_+(\blf,\blg))$.}
  \probleminput{$\delta$-parametric polynomials $\blf$ and $\blg$.}
  \problemquestion{Is there $\delta\ge 0$ such that $\blf_\delta\mid \blg_\delta$?}
\end{algproblem}

$\delta\ge 0$ satisfying $\blf_\delta\mid \blg_\delta$ is called
a \emph{witness} for the instance $(\blf,\blg)$ of $\DIV_+$.
For a $\delta$-parametric polynomial $\blf=\sum_{i=s}^t f_i(z) z^{i\delta}$
define another $\delta$-parametric polynomial
$$
\flip(\blf) = \sum_{i=s}^t f_i(z) z^{-i\delta}
$$
in which every $\delta$ is replaced with $-\delta$.
Clearly, $[\flip(\blf)]_\delta = \blf_{-\delta}$ for every $\delta\in\MZ$.
Therefore,
\begin{align*}
\exists \delta < 0 \mbox{ s.t. } 
\blf_\delta \mid \blg_\delta
&\ \ \Leftrightarrow\ \ 
\exists \delta < 0 \mbox{ s.t. } 
[\flip(\blf)]_{-\delta} \mid [\flip(\blg)]_{-\delta}\\
&\ \ \Leftrightarrow\ \ 
\exists \delta > 0 \mbox{ s.t. } 
[\flip(\blf)]_{\delta} \mid [\flip(\blg)]_{\delta}.
\end{align*}
It will be convenient to assume that 
$\DIV(\blf,\blg))$ and $\DIV_+(\blf,\blg))$ define a true/false value
depending on whether $(\blf,\blg)$ are positive/negative instances 
of the problem. As a corollary we have the following proposition.

\begin{proposition}\label{pr:div-div2}
$\DIV(\blf,\blg) = \DIV_+(\blf,\blg) \vee \DIV_+(\flip(\blf),\flip(\blg))$.
\end{proposition}

Without loss of generality,
to discuss the divisibility problem for \emph{$\delta$-parameter polynomials},
we may assume that the parameter $\delta \geq 0$.

\begin{proposition}\label{prop:main-equivalence}
A one-variable-equation $w(a,t,x)=1$
satisfying $\sum_{i=1}^k \varepsilon_i = 0$
has a solution if and only if the following conditions are satisfied:
\begin{itemize}
\item 
$\sum_{i=0}^k \delta_i = 0$
\item 
$(\den(w),\num(w))$ is a positive instance of the divisibility problem.
\end{itemize}
\end{proposition}

\begin{proof}
When $\sum_{i=1}^k \varepsilon_i = 0$ conditions (C1) and (C2)
of Lemma \ref{le:C1-C2} translate to the given conditions, respectively.
\end{proof}

For a $\delta$-parametric polynomial $\blf = \sum_{i=s}^t f_i(z) z^{i\delta}$
(where $f_s\ne0$, $f_t\ne0$)
define its \emph{$\delta$-degree} and \emph{$\delta$-order} as
$$
\deg_\delta(\blf) = t\ \ \mbox{ and }\ \ \ord_\delta(\blf) = s.
$$

\section{Properties of $\num_{\delta}$ and $\den_{\delta}$}
\label{se:num-den}

\subsection{Inductive formulae for $\num$ and $\den$}

Consider an equation $w(a,t,x)=1$ and the \emph{compact form} 
of the expression \eqref{eq:rhs} for the word $w$
\begin{equation}\label{eq:rhs-short}
(\delta_0,f_0)
(\delta,f)^{\varepsilon_1}
\dots
(\delta,f)^{\varepsilon_k}
(\delta_k,f_k) 
\ =\ 
(t_w + \delta x_w,\ \num(w) + f \den(w)).
\end{equation}
The polynomials $\den(w)$ and $\num(w)$ can be efficiently constructed 
by reading the word $w$ \textbf{from the right to the left}
as described by the following lemmas.

\begin{lemma}\label{le:p-num-den0}
For the empty word $\varepsilon$ we have
$x_\varepsilon=t_\varepsilon=0$ and 
$\num(\varepsilon) = \den(\varepsilon) =0$.
\end{lemma}

\begin{lemma}\label{le:p-num-den1}
If $w'=x\circ w$, then 
$t_{w'}=t_w$, $x_{w'}=x_w+1$,
$\num(w')=\num(w)$, and
$\den(w')=\den(w) + z^{-t_w-x_w \delta}$.
\end{lemma}

\begin{proof}
Multiplying the formula \eqref{eq:rhs-short} 
on the left by $x=(\delta,f)$ we get
$$
(\delta,f)\cdot(t_w + \delta x_w,\ \num(w) + f \den(w)) =
(\underbrace{t_w}_{t_{w'}} + \delta (\underbrace{x_w+1}_{x_{w'}}),\ 
\underbrace{\num(w)}_{\num(w')} + 
f(\underbrace{\den(w) +z^{-t_w - \delta x_w}}_{\den(w')}))
$$
which gives the claimed formulae.
\end{proof}

\begin{lemma}\label{le:p-num-den2}
If $w'= x^{-1} \circ w$, then 
$t_{w'} = t_w$, $x_{w'} = x_w -1$,
$\num(w')=\num(w)$, and
$\den(w')=\den(w) + z^{-t_w- (x_w-1) \delta}$.
\end{lemma}

\begin{proof}
Multiplying the formula \eqref{eq:rhs-short} 
on the left by $x^{-1}=(-\delta,-f^{-\delta})$ we get
$$
(-\delta,-f^{-\delta})\cdot(t_w + \delta x_w,\ \num(w) + f \den(w)) =
(t_w + \delta (x_w- 1),\ \num(w) + f(\den(w) +z^{-t_w - \delta (x_w - 1)}))
$$
which gives the claimed formulae.
\end{proof}

\begin{lemma}\label{le:p-num-den3}
If $w'=t\circ w$, then 
$t_{w'}=t_w+1$, 
$x_{w'}=x_w$,
$\num(w')=\num(w)$, and
$\den(w')=\den(w)$.
\end{lemma}

\begin{proof}
Multiplying the formula \eqref{eq:rhs-short} 
on the left by $t=(1,\blo)$ we get
$$
(1,\blo)\cdot(t_w + \delta x_w,\ \num(w) + f \den(w)) =
(t_w+1 + \delta x_w,\ \num(w) + f\den(w))
$$
which gives the claimed formulae.
\end{proof}

\begin{lemma}\label{le:p-num-den4}
If $w'= t^{-1} \circ w$, then 
$t_{w'}=t_w-1$, 
$x_{w'}=x_w$,
$\num(w')=\num(w)$, and
$\den(w')=\den(w)$.
\end{lemma}

\begin{proof}
Multiplying the formula \eqref{eq:rhs-short} 
on the left by $t^{-1}=(-1,\blo)$ we get
$$
(-1,\blo)\cdot(t_w + \delta x_w,\ \num(w) + f \den(w)) =
(t_w- 1 + \delta x_w,\ \num(w) + f\den(w))
$$
which gives the claimed formulae.
\end{proof}

\begin{lemma}\label{le:p-num-den5}
If $w' = a \circ w$, then 
$t_{w'}=t_w$,
$x_{w'}=x_w$,
$\num(w')=\num(w)+ z^{-t_w-x_w \delta}$, and
$\den(w')=\den(w)$.
\end{lemma}

\begin{proof}
Multiplying the formula \eqref{eq:rhs-short} 
on the left by $a=(0,\one_0)$ we get
$$
(0,\one_0)\cdot (t_w + \delta x_w,\ \num(w) + f \den(w)) =
(t_w + \delta x_w,\ (\num(w) +z^{-t_w-x_w \delta}) + f \den(w)) 
$$
which gives the claimed formulae.
\end{proof}

\subsection{Magnus-type embedding}

Consider the free abelian group $\MZ^2$ defined as 
the set $\Set{t^ix^j}{i,j\in\MZ}$ equipped with the binary operation 
$\cdot$ defined by $t^ix^j \cdot t^kx^l=t^{i+k}x^{j+l}$.
Consider the group ring $\MZ_2[\MZ^2]$; it is the set of all 
finite linear combinations of elements $t^ix^j$ from $\MZ^2$
with coefficients from $\MZ_2$. 

By definition, a \emph{$\delta$-parametric polynomial} with coefficients from $\MZ_2$ 
can be written as
a sum of finitely many monomials $\sum_{i=1}^n z^{a_i+b_i \delta}$, where $a_i,b_i\in\MZ$.
In fact, the set of all \emph{$\delta$-parametric polynomials} forms a ring,
and there is a ring isomorphism 
from the set of all \emph{$\delta$-parametric polynomials}
to the group ring $\MZ_2[\MZ^2]$ defined by
$$
z^{a+b \delta} \mapsto t^{a} x^{b}.
$$
Hence, all parametric polynomials can be represented 
by elements of $\MZ_2[\MZ^2]$.
Let $M$ be the group of $3\times 3$ lower triangular matrices of the form
$$
\left[
\begin{array}{ccc}
p & 0 & 0 \\
\blf & 1 & 0 \\
\blg & 0 & 1 \\
\end{array}
\right]
$$
where $p \in \MZ^2$ and $\blf,\blg\in \MZ_2[\MZ^2]$.
Define a homomorphism $\alpha: F(a,t,x) \to M$ by the map
\[
x \mapsto
X = 
\begin{bmatrix}
x^{-1} & 0 & 0 \\
0 & 1 & 0 \\
1 & 0 & 1 
\end{bmatrix},
\ \ \ \ 
t \mapsto
T =
\begin{bmatrix}
t^{-1} & 0 & 0 \\
0 & 1 & 0 \\
0 & 0 & 1 \\
\end{bmatrix},
\ \ \ \ 
a \mapsto
A = 
\begin{bmatrix}
1 & 0 & 0 \\
1 & 1 & 0 \\
0 & 0 & 1 
\end{bmatrix}
\]

\begin{proposition}\label{pr:num-den-matrices}
$\alpha(w)=
\left[
\begin{array}{ccc}
x^{-x_w}t^{-t_w} & 0 & 0 \\
\num(w) & 1 & 0 \\
\den(w) & 0 & 1 \\
\end{array}
\right]
$ for every $w\in F(a,t,x)$.
\end{proposition}

\begin{proof}
Induction on $|w|$.
If $|w| = 0$, then $\alpha(w)=I$ which agrees with Lemma \ref{le:p-num-den0}.
Then using the induction assumption, we deduce the following:
{\allowdisplaybreaks
\begin{align*}
\alpha(x\circ w)
&=
\begin{bmatrix}
x^{-1} & 0 & 0 \\
0 & 1 & 0 \\
1 & 0 & 1 
\end{bmatrix}
\cdot
\begin{bmatrix}
x^{-x_w}t^{-t_w} & 0 & 0 \\
\num(w) & 1 & 0 \\
\den(w) & 0 & 1 \\
\end{bmatrix}
=
\begin{bmatrix}
x^{-(x_w+1)}t^{-t_w} & 0 & 0 \\
\num(w) & 1 & 0 \\
\den(w)+x^{-x_w}t^{-t_w} & 0 & 1 \\
\end{bmatrix}\\
\alpha(t\circ w)
&=
\begin{bmatrix}
t^{-1} & 0 & 0 \\
0 & 1 & 0 \\
0 & 0 & 1 \\
\end{bmatrix}
\cdot
\begin{bmatrix}
x^{-x_w}t^{-t_w} & 0 & 0 \\
\num(w) & 1 & 0 \\
\den(w) & 0 & 1 \\
\end{bmatrix}
=
\begin{bmatrix}
x^{-x_w}t^{-(t_w+1)} & 0 & 0 \\
\num(w) & 1 & 0 \\
\den(w) & 0 & 1 \\
\end{bmatrix}\\
\alpha(a\circ w)
&=
\begin{bmatrix}
1 & 0 & 0 \\
1 & 1 & 0 \\
0 & 0 & 1 
\end{bmatrix}
\cdot
\begin{bmatrix}
x^{-x_w}t^{-t_w} & 0 & 0 \\
\num(w) & 1 & 0 \\
\den(w) & 0 & 1 \\
\end{bmatrix}
=
\begin{bmatrix}
x^{-x_w}t^{-t_w} & 0 & 0 \\
\num(w)+x^{-x_w}t^{-t_w} & 1 & 0 \\
\den(w) & 0 & 1 \\
\end{bmatrix}\\
\alpha(x^{-1}\circ w)
&=
\begin{bmatrix}
x & 0 & 0 \\
0 & 1 & 0 \\
x & 0 & 1 
\end{bmatrix}
\cdot
\begin{bmatrix}
x^{-x_w}t^{-t_w} & 0 & 0 \\
\num(w) & 1 & 0 \\
\den(w) & 0 & 1 \\
\end{bmatrix}
=
\begin{bmatrix}
x^{-(x_w -1)}t^{-t_w} & 0 & 0 \\
\num(w) & 1 & 0 \\
\den(w) + x^{-(x_w -1)}t^{-t_w} & 0 & 1 \\
\end{bmatrix}\\
\alpha(t^{-1}\circ w)
&=
\begin{bmatrix}
t & 0 & 0 \\
0 & 1 & 0 \\
0 & 0 & 1 \\
\end{bmatrix}
\cdot
\begin{bmatrix}
x^{-x_w}t^{-t_w} & 0 & 0 \\
\num(w) & 1 & 0 \\
\den(w) & 0 & 1 \\
\end{bmatrix}
=
\begin{bmatrix}
x^{-x_w}t^{-(t_w -1)} & 0 & 0 \\
\num(w) & 1 & 0 \\
\den(w) & 0 & 1 \\
\end{bmatrix}
\end{align*}}
In every case, the expressions in the first column coincide with the formulae from
Lemmas \ref{le:p-num-den1}, \ref{le:p-num-den3}, \ref{le:p-num-den5}, \ref{le:p-num-den2} and \ref{le:p-num-den4}.
Hence, the statement holds.
\end{proof}

\begin{comment}
We introduce a tool to enhance our intuition about the
polynomials $\num_{\delta}(w)$ and $\den_{\delta}(w)$ associated with a given word $w = w(a,t,x)$.
We can view the $xt$-grid as the Cayley graph of $\MZ^2 = \Set{t^ix^j}{i,j\in\MZ}$
in which every edge is labeled with $x^\pm,t^\pm$.

\begin{itemize}
\item The slope of the origin and terminus of the trace $(x_w, t_w)$ defines the value of $\delta = -\frac{t_w}{x_w}$;
\item The numerator polynomial $\num_{\delta}(w)$ is generated by accumulation of the positions of $a$-symbols under the isomorphism $(x_w, t_w) \mapsto t^{-t_w} x^{-x_w} \mapsto z^{-t_w - x_w \delta}$.
\item The denominator polynomial $\den_{\delta}(w)$ is defined by the positions of left endpoints of $x$-edges under the isomorphism $(x_w, t_w) \mapsto t^{-t_w} x^{-x_w} \mapsto z^{-t_w - x_w \delta}$.
\end{itemize}

\end{comment}

\subsection{Tracing equation $w(a,t,x)$ on the $xt$-grid}
\label{se:tracing}

Define the $xt$-grid as the Cayley graph of 
$\MZ^2 = \Set{t^ix^j}{i,j\in\MZ}$
in which every edge is labeled with $x^\pm,t^\pm$. 
That makes the $xt$-grid an automaton with infinitely many states
and with the initial state $x^0t^0$.
The next statement is obvious.

\begin{lemma}
There is a one-to-one correspondence between
$\delta$-parametric polynomials $\blf$ and
finite sets $S_{f}\subset \MZ^2$ given by the formula
$\blf = \sum_{x^it^j \in S_{f}} z^{j+i\delta}$.
\end{lemma}

For $w=(a,t,x)$ define the finite sets $N_w,D_w \subset \MZ^2$ by
$$
N_w = S_{\num(w)}\ \mbox{ and }\ D_w = S_{\den(w)}.
$$
In this section, we discuss a geometric procedure that 
we call the \emph{word-tracing},
that allows to construct $N_w$ and $D_w$ by following the path
on the $xt$-grid, defined by $w$, denoted by $\Path(w)$,
and adding points to $N_w$ and $D_w$ as we go.

Recall that the matrix formulas from the proof of
Proposition \ref{pr:num-den-matrices} compute 
the triple \\
$((-x_w,-t_w),\num_\delta(w),\den_\delta(w))$,
processing a given word $w$ letter by letter from the right to the left.
That translates into a geometric walk that follows a path $\Path(w)$
defined by inductively by the following formulas
($\SD$ denotes the symmetric difference of sets)
\begin{itemize}
\item 
For $w=\varepsilon$ we have 
\begin{itemize}
\item 
$\Path(\varepsilon)$ is the empty path that starts and ends at $x^0t^0$,
\item 
coordinates of the current point are $x_\varepsilon=t_\varepsilon=0$, 
\item 
$N_\varepsilon=\varnothing$ and $D_\varepsilon=\varnothing$.
\end{itemize}
\item 
For $w'=x\circ w$
\begin{itemize}
\item 
$\Path(w')=\Path(w)\circ e$, where
$e=x^{-x_w}t^{-t_w}\to x^{-(x_w+1)}t^{-t_w}$;
\item 
$x_{w'}=x_{w}+1$ and $t_{w'}=t_{w}$;
\item 
$D_{w'} = D_w \SD \{x^{-x_w}t^{-t_w}\}$ and $N_{w'}=N_w$.
\end{itemize}
\item 
For $w'=x^{-1}\circ w$
\begin{itemize}
\item 
$\Path(w')=\Path(w)\circ e$,
where $e=x^{-x_w}t^{-t_w} \to x^{-(x_w-1)}t^{-t_w}$;
\item 
$x_{w'}=x_{w}-1$ and $t_{w'}=t_{w}$;
\item 
$D_{w'} = D_w \SD \{x^{-(x_w-1)}t^{-t_w}\}$ and $N_{w'}=N_w$.
\end{itemize}
\item 
%==========================
For $w'=t\circ w$
\begin{itemize}
\item 
$\Path(w')=\Path(w)\circ e$,
where $e=x^{-x_w}t^{-t_w} \to x^{-x_w}t^{-(t_w+1)}$;
\item 
$x_{w'}=x_{w}$ and $t_{w'}=t_{w}+1$,
\item 
$D_{w'} = D_w$ and $N_{w'}=N_w$.
\end{itemize}
%==========================
\item 
For $w'=t^{-1}\circ w$
\begin{itemize}
\item 
$\Path(w')=\Path(w)\circ e$,
where $e=x^{-x_w}t^{-t_w} \to x^{-x_w}t^{-(t_w-1)}$;
\item 
$x_{w'}=x_{w}$ and $t_{w'}=t_{w}-1$,
\item 
$D_{w'} = D_w$ and $N_{w'}=N_w$.
\end{itemize}
%==========================
\item 
For $w'=a^{\pm}\circ w$
\begin{itemize}
\item 
$\Path(w')=\Path(w)$;
\item 
$x_{w'}=x_{w}$ and $t_{w'}=t_{w}$,
\item 
$D_{w'} = D_w$ and $N_{w'}=N_w \SD \{x^{-x_w}t^{-t_w}\}$.
\end{itemize}
\end{itemize}
To summarize, tracing $w$ letter-by-letter right to left on the $xt$-grid
\begin{itemize}
\item 
reading $x/x^{-1}$ we make a step left/right and add the current/right
point to $D_w$ (modulo two) resp.,
\item 
reading $t/t^{-1}$ we make a step down/up resp.,
\item 
reading $a/a^{-1}$ we stay at the current point and add the new
point to $N_w$ (modulo two).
\end{itemize}

\begin{figure}[h]
\centering
\includegraphics[width=0.3\linewidth]{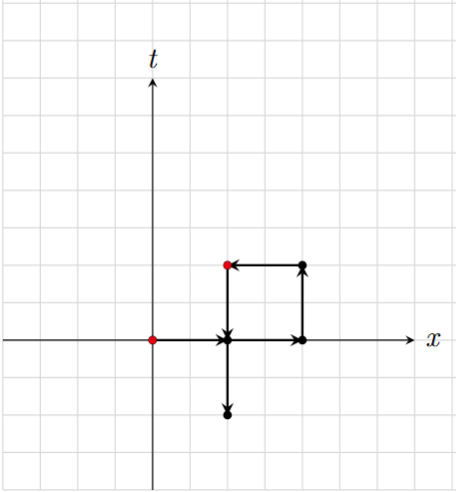}
\hspace{0.9cm}
\includegraphics[width=0.3\linewidth]{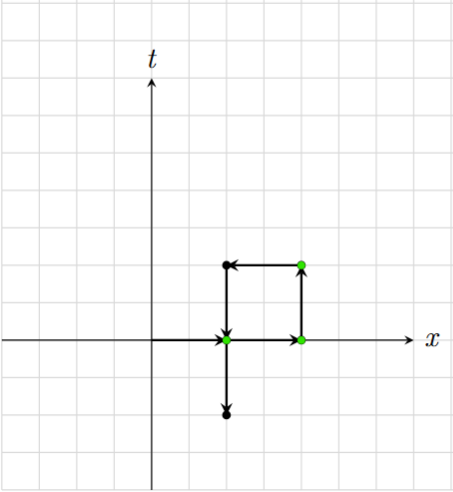}

\caption{Tracing the word $w= t^{2} a x t^{-1} x^{-2} a$ on the $xt$-grid produces the sets $N_w = \{ x^0 t^0, x^{1} t^{1} \}$ and $D_w = \{ x^{1}t^0, x^{2}t^0, x^{2}t^{1} \}$ for the polynomials $\num(w)$ and $\den(w)$.
The left and right diagrams show $N_w$ (red points) and $D_w$ (green points) resp.}
\label{fig:tracing_example}
\end{figure}

\begin{lemma}\label{le:square}
$N_w,D_w\ \subseteq\ [-|w|,|w|]^2$ for every $w=w(a,t,x)$.
\end{lemma}

\begin{proof}
Because tracing a word $w$ of length $|w|$ does not take us outside
of the square $[-|w|,|w|]^2$.
\end{proof}

%%%%%%%%%%%%%%%%%%%%%%%%%%%%%%%%%%
%%%%%%%%%%%%%%%%%%%%%%%%%%%%%%%%%%
%%%%%%%%%%%%%%%%%%%%%%%%%%%%%%%%%%

\subsection{Complexity of solving the equation $w(a,t,x)=1$ when $\sum \varepsilon_i \ne 0$}\label{subsection: complexity}

Suppose that the equation $w(a,t,x)=1$ satisfies $\sigma_x(w)\ne 0$.
Then the value of $\delta$ associated with the word $w(a,t,x)$ 
is defined by \eqref{eq:delta}.
Geometrically, $\delta=-\tfrac{t_w}{x_w}$ corresponds to
the negative slope of the line passing through the origin
$(0,0)$ and the terminus $(-x_w,-t_w)$ of $\Path(w)$, 
which implies the following.

\begin{lemma}\label{lemma:|w|}
$\sigma_x(w)\ne 0\ \ \Rightarrow\ \ 1\le |\delta| \le |w|$.
\end{lemma}

\begin{proof}
Because $|x_w|,|t_w|\le |w|$.
\end{proof}

Geometrically, substitution of the value of $\delta$ 
into $\num(w)$ and $\den(w)$ 
corresponds to projecting the points
$N(w)$ and $D(w)$ onto the $t$-axis
along the direction of the vector $\Vec{v} = (1, -\delta)$.

\begin{figure}[H]
\centering
\includegraphics[width=0.4\linewidth]{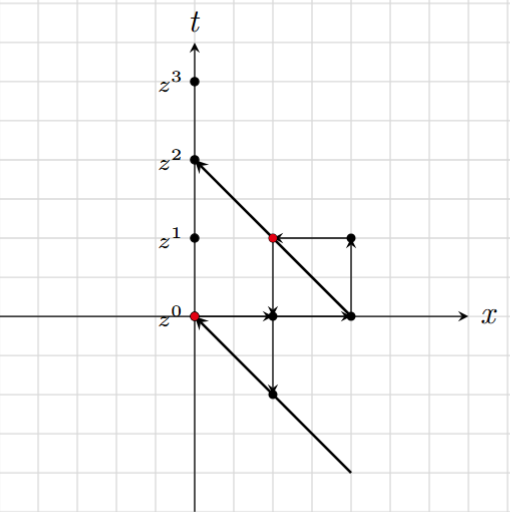}
\hspace{1cm}
\includegraphics[width=0.4\linewidth]{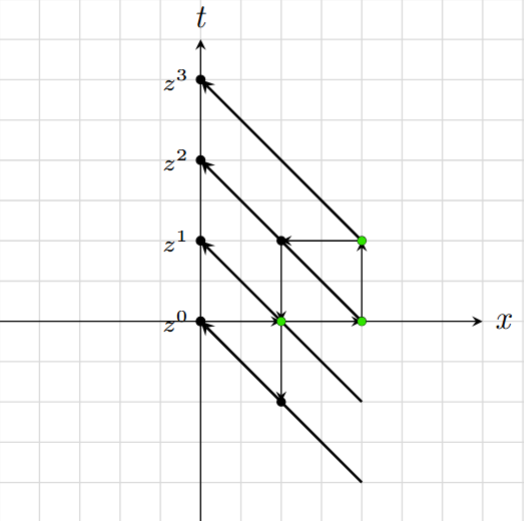}
\caption{Substituting $\delta \leftarrow 1 $ to the example in Figure \ref{fig:tracing_example}, we can visualize the numerator (see left diagram) and denominator (see right diagram) polynomials on $t$-axis.
The sets $N_w$ and $D_w$ are $N_w = \{ z^0, z^2 \}$ and $D_w =\{ z^1, z^2, z^3 \}$. }
\label{fig:proj_onto_t}
\end{figure}

\begin{corollary}\label{co:subst-deg-ord}
Suppose that $\sigma_x(w)\ne 0$. Then $\delta$ is defined by \eqref{eq:delta}
and the parametric polynomials $\num_{\delta}(w), \den_{\delta}(w)$
instantiated with $\delta$ satisfy
\begin{enumerate}
\item[(a)]
$\num_{\delta}(w)\ne 0
\ \ \Rightarrow\ \ 
-|w|^2 - |w| \le
|\ord(\num_{\delta}(w))|,
|\deg(\num_{\delta}(w))|
\le |w|^2+|w|.
$
\item[(b)]
$\den_{\delta}(w)\ne 0
\ \ \Rightarrow\ \ 
-|w|^2 - |w| \le
|\ord(\den_{\delta}(w))|,
|\deg(\den_{\delta}(w))|
\le |w|^2+|w|.
$
\end{enumerate}
\end{corollary}

\begin{proof}
Projecting the point $(x,t)\in N_w$ onto the $t$-axis along the direction 
of the vector $\Vec{v} = (1, -\delta)$ creates the point
$(0,t+\delta x)$, which contributes the monomial
$z^{t+\delta x}$ to $\num_\delta(w)$. 
Hence, the square $[-|w|,|w|]^2$ of Lemma \ref{le:square}
is projected onto the interval
$[-|w|^2-|w|,|w|^2+|w|]$.
Therefore, nontrivial $\num_{\delta}(w)$ satisfies
$$
-|w|^2 - |w| \le
\ord(\num_{\delta}(w)),
\deg(\num_{\delta}(w))
\le |w|^2+|w|.
$$
Same works for $\den(w)$.
\end{proof}

The bounds of Corollary \ref{co:subst-deg-ord} are asymptotically sharp, 
as illustrated by the following example.

\begin{example}
Let us consider the word $w = t^{1-n} x^{n-1} a t^{-1} x^{-n}a$ of the length $|w| = 3n+ 1$. 
Its trace and the polynomials $\num(w)$ and $\den(w)$
are visualized in Figure \ref{fig:quadratic_den_and_num}.
Its terminus is the point $(1,n)$,
$x_w=-1$, $t_w=-n$, and $\delta = -n$.
Replacing $\delta$ with $-n$ in $\num(w)$ and $\den(w)$ we get
$\num_{\delta}(w) = z^{-n^2 + 1} + z^0$ and
$\den_{\delta}(w) = (z^{-n} + \dots + z^{-n^2}) + (z^{-2n+1} + \dots + z^{-n^2 + 1})
= \sum_{i=1}^n z^{-in} + \sum_{i=2}^n z^{-in + 1}$.

\begin{figure}[h]
    \centering
    \includegraphics[width=0.45\linewidth]{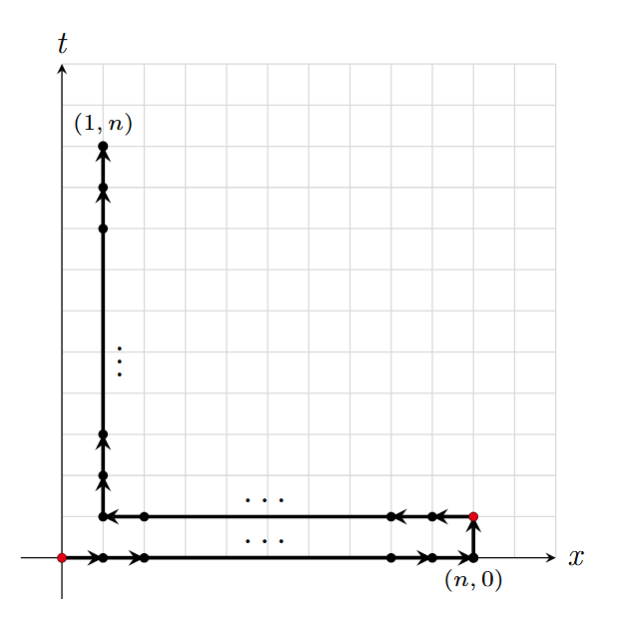}
    \hspace{0.5cm}
    \includegraphics[width=0.45\linewidth]{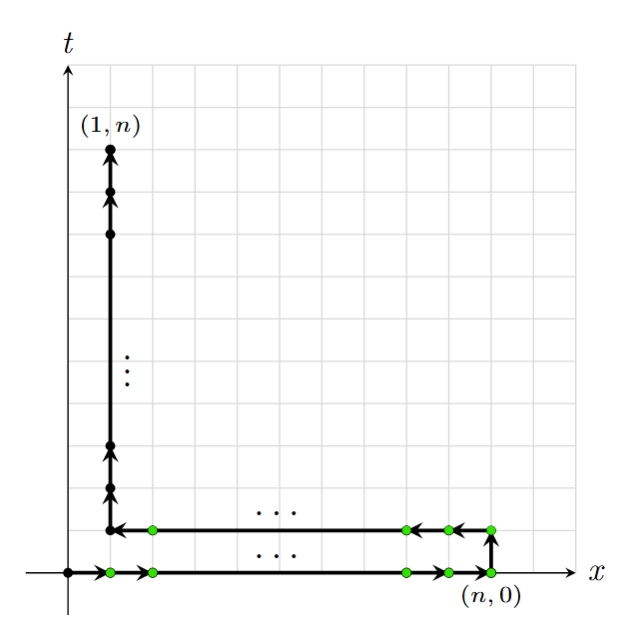}
\caption{Tracing the word $w = t^{1-n} x^{n-1} a t^{-1} x^{-n} a $ 
creates a path from $(0,0)$ to $(1,n)$ and produces the sets
$N_w =\{ x^n t^1, x^0 t^0 \}$ and 
$D_w =\{ x^1 t^0, x^2 t^0 \dots, x^n t^0, x^2 t^1, x^3 t^1, \dots , x^n t^1 \}$. 
}
    \label{fig:quadratic_den_and_num}
\end{figure}
\end{example}

\begin{proposition}\label{pr:compute-num-den}
If $\sigma_x(w) \ne 0$, then $\pres(\num_\delta(w))$ and 
$\pres(\den_\delta(w))$ can be computed in $O(|w|^2)$ time.
\end{proposition}

\begin{proof}
Representations for $\num_\delta(w)$ and $\den_\delta(w)$ 
are constructed in two steps.
First, compute $x_w = \sigma_x(w)$ and $t_w=\sigma_t(w)$
and use \eqref{eq:delta} to compute $\delta$ as $-t_w/x_w$.
Division of numbers bounded by $|w|$ can be done in $O(|w|)$ time.

Allocate memory for two functions $n,d:[-|w|^2-w,|w|^2+|w|] \to \{0,1\}$
to store the coefficients of $\num_\delta(w)$ or $\den_\delta(w)$ resp.
Then trace $w$ as described in Section \ref{se:tracing} and
instead of adding elements $x^{-x} t^{-t}$ to $N_w$
add $z^{t+\delta x}$ to $\num_\delta(w)$.
Notice that multiplication of $x$ by $\delta$ can be avoided
(replaced by adding/subtracting $\delta$ as $w$ is traced).
Finally, processing left-to-right form $\pres(\num_\delta(w))$.
The same works for $\pres(\den_\delta(w))$.
\end{proof}

\begin{proposition}\label{Prop:algorithm-equations}
There is a nearly quadratic $O(|w|^2 \log |w| \log \log |w|)$ 
time algorithm that for a given one-variable equation $w(a,t,x)=1$
satisfying $\sigma_x(w)\ne 0$, decides if it has a solution or not.
\end{proposition}

\begin{proof}
By Proposition \ref{pr:compute-num-den}, when $\sigma_x(w)\ne 0$,
we can compute $\delta$, $\num(w)$, and $\den(w)$ in $O(|w|^2)$ time,
or decide that $w=1$ has no solution in $O(|w|^2)$ time.
Check whether $\den_\delta(w) \mid \num_\delta(w)$.
By Proposition \ref{pr:Laurent-complexity}(d)
that requires $O((|w|^2+|w|) \log (|w|^2+|w|)^2 \log \log (|w|^2+|w|) )$
time. That can be simplified to $O(|w|^2 \log |w| \log \log |w|)$.
\end{proof}

\subsection{Generic-case complexity of the Diophantine problem for one-variable equations over $\MZ_2\wr\MZ$}

In this section we argue that a generic (typical) equation $w(a,t,x) = 1$
satisfies $\sigma_x(w)\ne 0$ and, hence, can be solved in nearly quadratic time.

Let us review several definitions related to generic-case complexity
\cite[Chapter 10]{MSU_book:2011}.
For $m\in\MN$ define the (finite) set
$$
S_{m} = \Set{w\in F(a,t,x)}{|w|=m},
$$
called the \emph{sphere} of radius $m$ in the Cayley graph of $F(a,t,x)$.
Let $P_m$ be the uniform probability distribution on $S_m$.
We say that a set $S\subseteq F(a,t,x)$ is generic in $F(a,t,x)$ if
$$
P_m(S\cap S_m) = \frac{|S\cap S_m|}{|S_m| } \to 1 \mbox{ as } m\to\infty,
$$
i.e., if the density of $S$ within the sets $S_m$
converges to $1$.
We say that a property $P(w)$ of elements $w\in F(a,t,x)$ is generic if the set 
$$
S=\Set{w}{P \mbox{ holds for } w}
$$
is generic. We say that a \emph{typical word} satisfies $P$
if $P$ is generic.

A uniformly random word $w(a,t,x)\in S_m$ of length $m$
can be generated as a sequence $w_1,\dots,w_m$ of
letters $\{a^\pm,t^\pm,x^\pm\}$ with no cancellations as follows:
\begin{itemize}
\item 
choose $w_1$ uniformly randomly from $\{a^\pm,t^\pm,x^\pm\}$;
\item 
choose $w_{i+1}$ uniformly randomly from
$\{a^\pm,t^\pm,x^\pm\} \setminus \{w_{i}^{-1}\}$ for $i\ge 1$.
\end{itemize}
% The chance for $w\in S_m$ to be generated by this procedure is
% $(6\cdot 5^{m-1})^{-1}$.
\begin{comment}
we can consider the generation of $S_{m+1}$ from $S_m$ by iteration,
which is similar to the Markov chain due to the independence.
\end{comment}
This type of generation procedure can be modeled by a discrete Markov chain
$\Gamma$ with seven states $Q=\{\varepsilon,a^\pm,t^\pm,x^\pm\}$,
the initial state $\varepsilon$, and with transitions 
\begin{itemize}
\item 
$\varepsilon \to q$, where $q\in \{a^\pm,t^\pm,x^\pm\}$ with uniform probability
$\tfrac{1}{6}$;
\item 
$q_1 \to q_2$, where $q_1\in \{a^\pm,t^\pm,x^\pm\}$ and 
$q_2 \in \{a^\pm,t^\pm,x^\pm\} \setminus\{q_1^{-1}\}$
with uniform probability $\tfrac{1}{5}$.
\end{itemize}
$\Gamma$ defines a sequence of random elements
$q_0,q_1,q_2,\dots \in Q$
(where $q_0=\varepsilon$ with probability $1$)
and for every $m\in\MN$ 
the corresponding (uniformly distributed) words $w=q_1\dots q_m \in S_m$.
Our goal is to analyze the value of $\sigma_x(w)$ for $w$ generated this way.
Notice the following: 
\begin{equation}\label{eq:sigma_x}
\sigma_x(q_1\dots q_m) = 
|\Set{1\le i\le m}{q_i=x}| - |\Set{1\le i\le m}{q_i=x^{-1}}|.
\end{equation}
Since $\Gamma$ has a lot of symmetries and since the states $a^\pm,t^\pm$ 
are not involved in the RHS of \eqref{eq:sigma_x},
we can identify the four states $a^\pm,t^\pm$ into a single state $y$,
which creates a new transition diagram shown below, denoted by $\Gamma^\ast$.
One can think that the new process generates words in which
the letters $a^\pm,t^\pm$ are replaced with the new letter $y$.

\begin{figure}[h]
\centering
\begin{subfigure}[t]{0.5\linewidth}
\includegraphics[width=\linewidth]{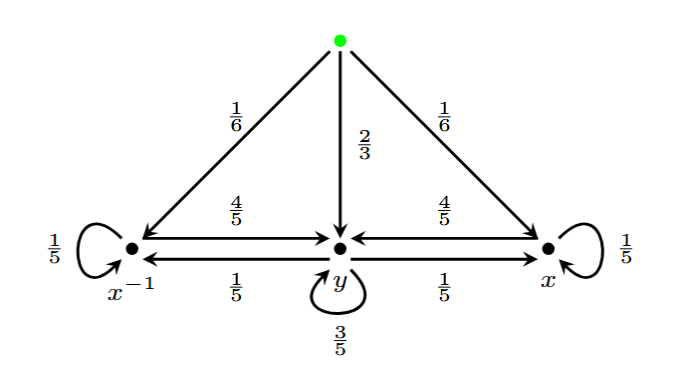}
\caption{Transition diagram $\Gamma^\ast$.}
\label{fig:transition probability with epsilon}
\end{subfigure}

\hfill

\begin{subfigure}[t]{0.5\linewidth}
\includegraphics[width=\linewidth]{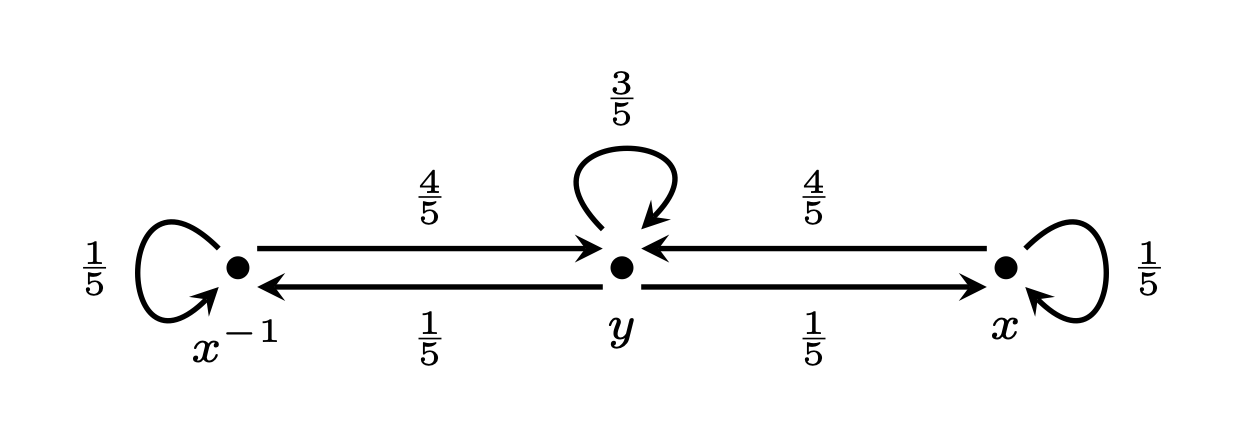}
\caption{Simplifying the diagram of $\Gamma^\ast$ to show the middle procedures of Markov chain $\Gamma$.}
\label{fig:transition probability}
\end{subfigure}
\end{figure}

\begin{comment}
% \begin{figure}[H]
%     \centering
%     \includegraphics[width=0.5\linewidth]{transition probability_2.png}
%     \caption{}
%     \label{fig:choice_1}
% \end{figure}

% \begin{figure}[H]
%     \centering
%     \includegraphics[width=0.5\linewidth]{transition probability_3.png}
%     \caption{}
%     \label{fig:choice_2}
% \end{figure}

\end{comment}

Notice that this process is not irreducible. However, deleting
the state $\varepsilon$ we get a new a Markov chain $q_1',q_2',\dots$ in which 
the initial distribution is defined by
$$
\Pr[q_1' = q] = 
\begin{cases}
1/6,& \mbox{ if } q=x^{-1} \\
2/3,& \mbox{ if } q=y \\
1/6,& \mbox{ if } q=x,
\end{cases}
$$
with the transition matrix $M$
$$
M=
\begin{bmatrix}
0.2 &  0.8 & 0\\
0.2 &  0.6 & 0.2 \\
0 &  0.8 & 0.2 \\
\end{bmatrix}
$$
The new Markov chain is \emph{irreducible}
(it is possible to reach any state from any other state 
with positive probability) and \emph{aperiodic}
(the greatest common divisor of return times to every state is $1$).
Hence, it is \emph{ergodic} and 
\begin{itemize}
\item 
it admits a unique \emph{stationary distribution} $\pi$ (a probability distribution 
satisfying $\pi M=\pi$), here $\pi$
is defined by the vector $(\tfrac{1}{6},\tfrac{2}{3},\tfrac{1}{6})$,
\item 
and starting from any initial distribution, the chain’s distribution 
converges to this stationary distribution as the number of steps goes 
to infinity.
\end{itemize}
It is well-known that for finite state spaces, ergodic Markov chains are \emph{uniformly ergodic}, i.e., it converges to its stationary distribution 
at a uniform (geometric) rate, independent of the starting state.
It is also \emph{Harris recurrent}: starting from any initial state, the chain will 
eventually visit every set of positive stationary measure infinitely 
often, with probability $1$.

Now consider the function $f:\{x^{-1},y,x\} \to \{-1,0,1\}$ defined by
$$
f(x^{-1})=-1,\ \ f(y)=0,\ \ f(x)=1.
$$
Clearly, $\ME_\pi[f] = 0$. By definition of $f$
\begin{equation}\label{eq:sigma_x2}
\sigma_x(q_1'\dots q_m') = f(q_1')+\dots+f(q_m').
\end{equation}
Define a random value (the time averaged estimator of $\ME_\pi[f]$)
$$
\bar{f}_m = \frac{1}{m} \sum_{i=1}^m f(q_i').
$$

\begin{theorem*}[Markov chain CLT, see {{\cite[Theorem 9(6)]{Jones:2004}}}]
If $X$ is a Harris ergodic Markov chain with
stationary distribution $\pi$, then 
$$
\sqrt{m}(\bar{f}_m - \ME_\pi[f])
\stackrel{d}{\longrightarrow} N(0,\sigma^2)
$$
if $X$ is uniformly ergodic and $\ME_\pi[f^2]<\infty$.
\end{theorem*}

For our Markov chain that translates into
$
\frac{1}{\sqrt{m}} \sum_{i=1}^m f(q_i')
\stackrel{d}{\longrightarrow} N(0,\sigma^2)
$, which implies the following.

\begin{proposition}
$\Pr[f(q_1')+\dots+f(q_m') = 0]\to 0$ as $m\to \infty$.
\end{proposition}

\begin{proof}
$\Pr[\sum_{i=1}^m f(q_i') = 0] = \Pr[\frac{1}{\sqrt{m}} \sum_{i=1}^m f(q_i') = 0]$
which, by the previous theorem, converges to the probability
that a normally distributed value equals $0$, which is zero.
\end{proof}

\begin{corollary}\label{co:sigma-x-generic}
$P_m[\Set{w\in S_m}{\sigma_x(w)=0}] \to 0 \mbox{ as } m\to \infty.$
\end{corollary}

\begin{corollary}
$\DP_1(L_2)$ has $O(|w|^2 \log |w| \log \log |w|)$ generic-case time complexity.
\end{corollary}

\begin{proof}
By Proposition \ref{Prop:algorithm-equations}
there is an algorithm that for a given equation $w=1$
decides if it has a solution in $O(|w|^2 \log |w| \log \log |w|)$ 
By Corollary \ref{co:sigma-x-generic}, the set of words
$\Set{w\in F(a,t,x)}{\sigma_x(w)\ne 0}$ is generic in $F(a,t,x)$.
\end{proof}

\section{Division-by-$f$ automaton}
\label{se:division-automaton}

Consider a polynomial $f\in\MZ_2[z]$.
Let $n=\deg(f)$.
The \emph{division-by-$f$ automaton} is a finite state automaton (FSA)
$\Gamma_f$ over the binary alphabet $\Sigma=\{0,1\}$
defined as follows.
\begin{itemize}
\item 
The set of states $S$ is the set of polynomials of degree less than $n$.
\item 
$0\in S$ is the starting and the accepting state.
\item 
The state transitions are defined by 
\begin{itemize}
\item 
$g\ \stackrel{0}{\longrightarrow}\ g\cdot z\ \%\ f$,
\item 
$g\ \stackrel{1}{\longrightarrow}\ g\cdot z+1\ \%\ f$,
\end{itemize}
for every $g\in S$, where $\%$ denotes the operation of 
computing the remainder of division.
\end{itemize}
Note that by definition, $\Gamma_f$ has a transition 
$s\stackrel{b}{\longrightarrow}t$ for some $s,t\in S$ and $b=0,1$
if and only if $t\equiv_f s\cdot z+b$.
Figure \ref{fi:Gamma-z3-z-1} shows $\Gamma_f$ for $f=z^3+z+1$.

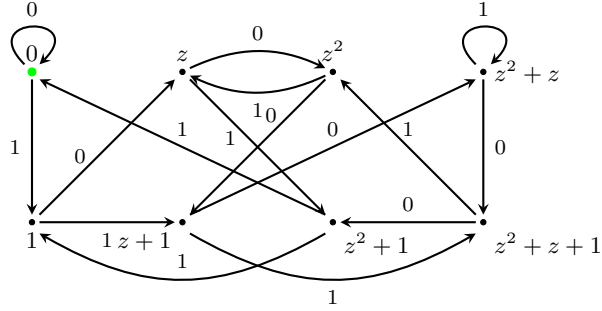
\begin{figure}[h]
\centering
\begin{tikzpicture}[>=stealth][scale=1.2]
% All nodes and vertices with the label
\filldraw[black] (0,0) circle (1pt) node[below][black]{\scriptsize $1$};
\filldraw[black] (2,0) circle (1pt) node[anchor=north east][black]{\scriptsize $z+1$};
\filldraw[black] (4,0) circle (1pt) node[anchor=north west][black]{\scriptsize $z^2 +1 $};
\filldraw[black] (6,0) circle (1pt) node[anchor=north west][black]{\scriptsize $z^2 + z+1$};

\filldraw[green] (0,2) circle (1.5pt) node[above][black]{\scriptsize $0$};
\filldraw[black] (2,2) circle (1pt) node[above][black]{\scriptsize $z$};
\filldraw[black] (4,2) circle (1pt) node[above][black]{\scriptsize $z^2$};
\filldraw[black] (6,2) circle (1pt) node[right][black]{\scriptsize $z^2 +z$};

% the basic edges to form the square
\draw[->, thick] [black] (0.1,0)--(1.9,0) node[midway, below][black]{\tiny $1$};
\draw[->, thick] [black] (0,1.9)--(0,0.1) node[midway, left][black]{\tiny $1$};
\draw[<-, thick] [black] (6,0.1)--(6,1.9) node[midway, right][black]{\tiny $0$};
\draw[->, thick] [black] (5.9,0)--(4.1,0) node[midway, above][black]{\tiny $0$};

% two cycles above
\draw[->, thick] (-0.1,2.1) to[out=135, in=45, looseness=12] node[midway, above][black]{\tiny $0$} (0.1,2.1);
\draw[->, thick] (5.9,2.1) to[out=135, in=45, looseness=12] node[midway, above][black]{\tiny $1$} (6.1,2.1);

% All gradient edges
\draw[->, thick] [black] (0.1,0.1)--(1.9,1.9) node[pos=0.3, above][black]{\tiny $0$};
\draw[->, thick] [black] (2.1,0.1)--(5.9,1.9) node[midway, above][black]{\tiny $0$};
\draw[->, thick] [black] (3.9,1.9)--(2.1,0.1) node[pos=0.4, above][black]{\tiny $0$};
\draw[->, thick] [black] (3.9,0.1)--(0.1,1.9) node[midway, above][black]{\tiny $1$};
\draw[->, thick] [black] (2.1,1.9)--(3.9,0.1) node[pos=0.3, below][black]{\tiny $1$};
\draw[->, thick] [black] (5.9,0.1)--(4.1,1.9) node[midway, above][black]{\tiny $1$};

% double edges for the reverse
\draw[->, thick] (2.1,2.05) to[out=25, in=155] node[midway, above][black]{\tiny $0$} (3.9,2.05);
\draw[->, thick] (3.9,1.95) to[out=-155, in=-25] node[midway, below][black]{\tiny $1$} (2.1,1.95);
% this part is for the curve of the double edges

% the curves under the square
\draw[->, thick] (2.1,-0.15) to[out=-30, in=-150, looseness=1.1] node[midway, below][black]{\tiny $1$} (5.9,-0.15);
\draw[->, thick] (3.9,-0.15) to[out=-150, in=-30, looseness=1.1] node[midway, above][black]{\tiny $1$} (0.1,-0.15);

\end{tikzpicture}
\caption{$\Gamma_f$ for $f=z^3+z+1$.}
\label{fi:Gamma-z3-z-1}
\end{figure}

\begin{lemma}\label{le:div-automaton}
The sequence of transitions 
$
s_0
\stackrel{g_m}{\longrightarrow}s_1
\stackrel{g_{m-1}}{\longrightarrow}s_2
\stackrel{g_{m-1}}{\longrightarrow}
\dots
\stackrel{g_{1}}{\longrightarrow}s_m
\stackrel{g_0}{\longrightarrow}s_{m+1}
$
in $\Gamma_f$
that starts at the state $s_0$ and is labeled with 
$g_m\dots g_0$ leads to the state 
$$
s_{m+1} = (s_0 z^{m+1} + g_m z^m + g_{m-1} z^{m-1} + \dots + g_1 z + g_0) \,\%\, f,
$$
or, equivalently, the following holds:
$$
s_0 z^{m+1} + g_m z^m + g_{m-1} z^{m-1} + \dots + g_1 z + g_0 + s_{m+1} \equiv_f 0.
$$
\end{lemma}

\begin{proof}
Induction on $m$.
By definition of transitions, the statement is true for $m=0$.
Suppose that the statement is true for every bit string of length
smaller than $m+1$.
Hence, for the given sequence of transitions, we have
$$
s_{m} \equiv_f s_0 z^{m} + g_m z^{m-1} + \dots + g_1,
$$
which implies that
\begin{align*}
s_{m+1} \equiv_f s_{m}z+g_0 &\equiv_f (s_0 z^{m} + g_m z^{m-1} + \dots + g_1)z +g_0\\
&= s_0 z^{m+1} + g_m z^m + \dots + g_1 z + g_0.
\end{align*}
\end{proof}

We say that $\Gamma_f$ \emph{accepts} a bit string $g_m\dots g_0$
if there is a sequence of transitions in $\Gamma_f$
from the starting state $0$ to the final state $0$
$$
0=s_0
\stackrel{g_m}{\longrightarrow}s_1
\stackrel{g_{m-1}}{\longrightarrow}s_2
\stackrel{g_{m-1}}{\longrightarrow}
\dots
\stackrel{g_{1}}{\longrightarrow}s_m
\stackrel{g_0}{\longrightarrow}s_{m+1}=0.
$$

\begin{corollary}
$\Gamma_f$ accepts $g_m\dots g_0$ $\ \ \Leftrightarrow\ \ $
$f \mid g_m z^m + g_{m-1} z^{m-1} + \dots + g_1 z + g_0$.
\end{corollary}

We say that a FSA $M$ is \emph{strongly connected} if for any states 
$s_i, s_j$ there is a sequence of transitions from $s_i$ to $s_j$.

\begin{lemma}
$\Gamma_f$ is strongly connected.
\end{lemma}

\begin{proof}
Pick any $g=g_{n-1}z^{n-1}+\dots+g_1z+g_0 \in S$
and consider the sequence of transitions
from $0$
$$
0=s_0
\stackrel{g_{n-1}}{\longrightarrow}s_1
\stackrel{g_{n-2}}{\longrightarrow}s_2
\stackrel{g_{n-3}}{\longrightarrow}
\dots
\stackrel{g_{1}}{\longrightarrow}s_{n-1}
\stackrel{g_0}{\longrightarrow}s_{n}.
$$
By Lemma \ref{le:div-automaton},
$s_{n}\in S$ is the remainder of division of
$$
g_{n-1}z^{n-1}+\dots+g_1z+g_0
$$
by $f$, which is $g$.
Hence, for any $g\in S$ there is a path from the initial state $0$ to $g$.

Conversely, let $g_{n-1}'z^{n-1}+\dots+g_1'z+g_0'$ be the remainder of
division of $g\cdot z^{n}$ by $f$. Then
$g\cdot z^{n} + g_{n-1}'z^{n-1}+\dots+g_1'z+g_0' \equiv_f 0$ and,
by Lemma \ref{le:div-automaton},
the sequence of transitions from $g$
$$
g=s_0
\stackrel{g_{n-1}'}{\longrightarrow}s_1
\stackrel{g_{n-2}'}{\longrightarrow}s_2
\stackrel{g_{n-3}'}{\longrightarrow}
\dots
\stackrel{g_{1}'}{\longrightarrow}s_{n-1}
\stackrel{g_0'}{\longrightarrow}s_{n}
$$
leads to $s_n=0$.
\end{proof}

Recall that a FSA over an alphabet $\Sigma$ 
is \emph{deterministic} if it has a unique
starting state and for every state $s$ and every $x\in\Sigma$
there is at most one transition that starts at $s$ labeled with $x$.
By construction, $\Gamma_f$ is deterministic. 
Furthermore, the following holds.

\begin{lemma}\label{le:unique-in}
If $f\not\equiv_z 0$, then 
for every $g\in V$ and every $x\in\{0,1\}$ there exists
a unique transition in $\Gamma_f$ labeled with $x$ that ends at $g$.
\end{lemma}

\begin{proof}
The assumption that $f\not\equiv_z 0$
implies that $\gcd(f,z)=1$ which ensures that the equations 
$$
h_0 \cdot z \equiv_f g
\ \ \mbox{ and }\ \ 
h_1\cdot z+1 \equiv_f g
$$
have unique solutions for $h_0$ and $h_1$ respectively.
Hence, $\Gamma_f$ has unique transitions
$h_0\ \stackrel{0}{\to}\ g$ and
$h_1\ \stackrel{1}{\to}\ g$
labeled with $0$ and $1$ that end at $g$.
\end{proof}

\subsection{Periodicity of division}

It follows from Lemma \ref{le:div-automaton}
that the sequence of intermediate states 
$s_0,\dots, s_{m+1}$ in the transition chain
$
s_0
\stackrel{g_m}{\longrightarrow}s_1
\stackrel{g_{m-1}}{\longrightarrow}s_2
\stackrel{g_{m-2}}{\longrightarrow}
\dots
\stackrel{g_{1}}{\longrightarrow}s_m
\stackrel{g_0}{\longrightarrow}s_{m+1}
$
is the sequence of successive remainders obtained by the long division 
of the polynomial
$s_0 z^{m+1} + g_m z^m + \dots + g_1 z + g_0$ by $f$.
Informally, we say that division by $f$ is periodic if 
the sequence of states (remainders) is periodic.

Formally, we say that a polynomial $g$ is \emph{$p$-periodic} if the sequence of
its coefficients $g_m,\dots,g_0$
satisfies $g_{i}=g_{i+p}$ for every $0\le i\le m-p$.
We do not require $p$ to be the minimal value satisfying that property.
In this section we prove that division of a $p$-periodic polynomial
is $pP$-periodic for some constant $P$ that depends only on $f$.

A more technical result that implies periodicity of division 
can be formulated as follows.
We claim that for divisor $f$ satisfying $f\not\equiv_z 0$, 
there exists $P=P_f\in\MN$ such that for every $s_0\in \Gamma_f$
and every bit string $w$ the transition chain that starts at $s_0$
labeled with $w^P$ ends at $s_0$.
Therefore, division of any $p$-periodic polynomial produces 
$pP$-periodic result.

First, let us introduce new symbols $0^{-1}, 1^{-1}$ to serve as 
formal inverses for
the symbols $0,1$ and consider the free group $F(0,1)$ of reduced
group-words over the alphabet $\Sigma=\{0,1\}$.
By Lemma \ref{le:unique-in}, $\Gamma_f$ is \emph{folded}
(\cite[Defintion 2.3]{Kapovich_Miasnikov:2002}).
Hence, $\Gamma_f$ can be viewed as a \emph{subgroup graph} 
(see \cite[Definition 5.3]{Kapovich_Miasnikov:2002}).
Denote by $H_f$ the subgroup of $F(0,1)$ defined by $\Gamma_f$.

\begin{proposition}\label{pr:periodic-division}
For every $f\not\equiv_z 0$ there exists $P=P_f\in\MN$ such that
for every $w\in\{0,1\}^\ast$ and every $s_0\in V(\Gamma_f)$
the transition chain in $\Gamma_f$ starting at $s_0$ and labeled by $w^P$
returns to $s_0$.
\end{proposition}

\begin{proof}
By construction, $\Gamma_f$ is regular 
(\cite[Defintion 8.1]{Kapovich_Miasnikov:2002}).
Hence, by \cite[Proposition 8.3]{Kapovich_Miasnikov:2002},
$\Gamma_f$ defines a finite index subgroup $H_f$ of $F(0,1)$, 
which is not normal in general
(e.g., $H_{z^2+1}$ is not normal by 
\cite[Theorem 8.14]{Kapovich_Miasnikov:2002}, see Figure \ref{fi:Gamma-z2-1}),
which can be normalized using the following procedure. 

Let $V(\Gamma_f)=\{v_1,\dots,v_{2^n}\}$.
For a vertex $v_i\in V(\Gamma_f)$,
denote by $\Gamma_{f,v_i}$ a copy of $\Gamma_f$ in which the root 
is set at $v_i$. 
Compute the \emph{product graph}
(see \cite[Definition 9.1]{Kapovich_Miasnikov:2002})
$$
\Gamma_{f,v_1}\times 
\Gamma_{f,v_2}\times 
\dots
\times \Gamma_{f,v_{2^n}};
$$
take its connected component containing $\ovv^\ast=(v_1,\dots,v_{2^n})$,
and designate $\ovv^\ast$ as the root.
The obtained graph $\Gamma^\ast$
defines the following normal subgroup of $F(0,1)$:
$$
N=\bigcap_{c\in F(0,1)} c^{-1} H_f c,
$$
which has finite index in $F(0,1)$.
Hence, $\Gamma^\ast$ is the Cayley graph of a finite group $G\simeq F(0,1)/N$. 
Set $P=|G|=|V(\Gamma^\ast)|$. By Lagrange theorem 
for every $g\in G$, $g^P=1$. Hence, for every $w=w(0,1)$,
the transition chain in $\Gamma^\ast$ 
starting at any $s_0'\in V(\Gamma^\ast)$ and labeled with $w^P$
returns to $s_0'$.

Finally, there is a natural subgroup graph epimorphism
$\varphi:\Gamma^\ast\to \Gamma_f$ that maps the root
$\ovv^\ast\in V(\Gamma^\ast)$ to the root $0 \in V(\Gamma_f)$.
Pick any $s_0\in V(\Gamma_f)$ and any $s_0'\in V(\Gamma^\ast)$
satisfying $\varphi(s_0')=s_0$. As we proved above,
for every $w=w(0,1)$
the transition chain in $\Gamma^\ast$ that starts at 
$s_0'\in V(\Gamma^\ast)$ and is labeled with $w^P$ ends at $s_0'$.
Apply $\varphi$ to that transition chain and obtain a chain
from $s_0$ to $s_0$ labeled with $w^P$ in $\Gamma_f$.
\end{proof}

\begin{figure}[h]
\centering
\begin{minipage}{0.4\textwidth}
\centering
\begin{tikzpicture}[scale=1.2]
\filldraw[black] (0,0) circle (1pt) node[left][black]{\scriptsize $1$};
\filldraw[green] (0,2) circle (1.5pt) node[left][black]{\scriptsize $0$};
\filldraw[black] (2,0) circle (1pt) node[right][black]{\scriptsize $z +1 $};
\filldraw[black] (2,2) circle (1pt) node[right][black]{\scriptsize $z $};

\draw[->, thick] [black] (0,1.9)--(0,0.1) node[midway, left][black]{\tiny $1$};
\draw[->, thick] [black] (0.1,0)--(1.9,0) node[midway, below][black]{\tiny $1$};
\draw[->, thick] [black] (1.9,2)--(0.1,2) node[midway, above][black]{\tiny $1$};
\draw[->, thick] [black] (2,0.1)--(2,1.9) node[midway, right][black]{\tiny $1$};

\draw[<->, double] [black] (0.1,0.1)--(1.9,1.9) node[midway, left][black]{\tiny $0$} node[midway, right][black]{\tiny $0$};

\draw[->, thick] (-0.1,2.1) to[out=135, in=45, looseness=10] node[midway, above][black]{\tiny $0$} (0.1,2.1);
\draw[->, thick] (2.1,-0.1) to[out=-45, in=-135, looseness=10] node[midway, below][black]{\tiny $0$} (1.9,-0.1);

\end{tikzpicture}
\end{minipage}
\hspace{1cm}
\begin{minipage}{0.4\textwidth}
\begin{tikzpicture}[scale=0.95]
\filldraw[black] (0,0) circle (1pt) node[below][black]{};
\filldraw[black] (2,0) circle (1pt) node[below][black]{};
\filldraw[black] (4,0) circle (1pt) node[below][black]{};
\filldraw[black] (6,0) circle (1pt) node[below][black]{};

\filldraw[green] (0,2) circle (1.5pt) node[above][black]{};
\filldraw[black] (2,2) circle (1pt) node[above][black]{};
\filldraw[black] (4,2) circle (1pt) node[above][black]{};
\filldraw[black] (6,2) circle (1pt) node[above][black]{};

\draw[<-, thick] [black] (0.1,0)--(1.9,0) node[midway, below][black]{\tiny $1$};
\draw[<-, thick] [black] (2.1,0)--(3.9,0) node[midway, below][black]{\tiny $1$};
\draw[<-, thick] [black] (4.1,0)--(5.9,0) node[midway, below][black]{\tiny $1$};

\draw[->, thick] [black] (0.1,2)--(1.9,2) node[midway, above][black]{\tiny $1$};
\draw[->, thick] [black] (2.1,2)--(3.9,2) node[midway, above][black]{\tiny $1$};
\draw[->, thick] [black] (4.1,2)--(5.9,2) node[midway, above][black]{\tiny $1$};

\draw[<->, double] [black] (0,0.1)--(0,1.9) node[midway, left][black]{\tiny $0$} node[midway, right][black]{\tiny $0$};
\draw[<->, double] [black] (2,0.1)--(2,1.9) node[midway, left][black]{\tiny $0$} node[midway, right][black]{\tiny $0$};
\draw[<->, double] [black] (4,0.1)--(4,1.9) node[midway, left][black]{\tiny $0$} node[midway, right][black]{\tiny $0$};
\draw[<->, double] [black] (6,0.1)--(6,1.9) node[midway, left][black]{\tiny $0$} node[midway, right][black]{\tiny $0$};

\draw[->, thick] (0.1,-0.1) to[out=-30, in=-150, looseness=1.1] node[midway, below][black]{\tiny $1$} (5.9,-0.1);
\draw[->, thick] (5.9,2.1) to[out=150, in=30, looseness=1.1] node[midway, above][black]{\tiny $1$} (0.1,2.1);
\end{tikzpicture}
\end{minipage}

\caption{$\Gamma_f$ for $f=z^2+1$ and its normalization 
which is the Cayley graph of $D_4\simeq \MZ_4\rtimes \MZ_2$ with $P_f=8$.}
\label{fi:Gamma-z2-1}
\end{figure}
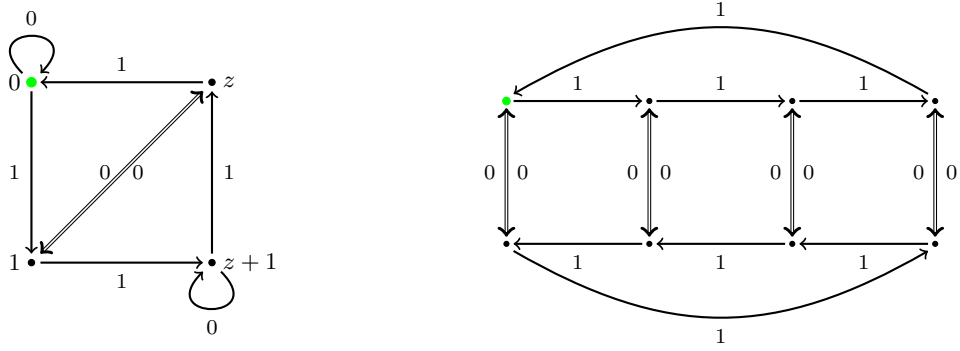

Let us examine some properties of the product graph
$\Gamma_{f,v_1}\times \dots \times \Gamma_{f,v_{2^n}}$
discussed in the proof of 
Proposition \ref{pr:periodic-division}.
By definition, its vertex set is the Cartesian product
$V(\Gamma_{f,v_1})\times \dots \times V(\Gamma_{f,v_{2^n}})$, i.e.,
its vertices are the tuples $\ovu=(u_1,\dots,u_{2^n})$ 
of vertices $u_i\in V(\Gamma_f)$.
For a tuple $\ovu=(u_1,\dots,u_{2^n})$ and a polynomial $h(z)$ define
the tuples $\ovu\cdot h(z)$ and $\ovu + h(z)$ by
\begin{align*}
\ovu\cdot h(z) &= (u_1\cdot h(z) \,\%\, f,\dots,u_{2^n}\cdot h(z) \,\%\, f),\\
\ovu + h(z) &= (u_1+h(z) \,\%\, f,\dots,u_{2^n}+ h(z) \,\%\, f).
\end{align*}
Then the transitions of $\Gamma_{f,v_1}\times \dots \times \Gamma_{f,v_{2^n}}$
can be defined algebraically as follows:
\begin{itemize}
\item 
$\ovu\stackrel{0}{\longrightarrow} \ovu\cdot z$, and
\item 
$\ovu\stackrel{1}{\longrightarrow} \ovu\cdot z+1$.
\end{itemize}

\begin{lemma}\label{le:div-automaton2}
The sequence of transitions 
$
\ovv^\ast
\stackrel{b_m}{\longrightarrow} \ovv_1
\stackrel{b_{m-1}}{\longrightarrow} \ovv_2
\stackrel{b_{m-1}}{\longrightarrow}
\dots
\stackrel{b_{1}}{\longrightarrow} \ovv_m
\stackrel{b_0}{\longrightarrow} \ovv_{m+1}
$
in $\Gamma^\ast$
that starts at $\ovv^\ast$ and is labeled with 
$b_m\dots b_0$ leads to the state 
$$
\ovv_{m+1} =
\ovv^\ast z^{m+1} + b_m z^m + b_{m-1} z^{m-1} + \dots + b_1 z + b_0.
$$
\end{lemma}

\begin{proof}
Follows from Lemma \ref{le:div-automaton}.
\end{proof}

\begin{lemma}\label{le:P-bound}
$\deg(f)=n\ \ \Rightarrow\ \ P_f\le 4^n$.
\end{lemma}

\begin{proof}
We claim that the connected component of 
$
\Gamma_{f,v_1}\times 
\dots
\times \Gamma_{f,v_{2^n}}
$
containing $\ovv^\ast$ has at most $4^n$ elements.
Indeed, by Lemma \ref{le:div-automaton2},
any sequence of transitions that starts at $\ovv^\ast$
labeled with $b_m \dots b_0$ leads to the state 
$$
\ovv = \ovv^\ast \cdot z^{m+1} + (b_m z^m + \dots +b_1 z+b_0)\,\%\, f.
$$ 
The number of different tuples of the form $\ovv^\ast \cdot z^{m+1}$
is not greater than the number of distinct powers $z^m$ modulo $f$, 
which is less than $2^n$.
The number of distinct terms $(b_m z^m + \dots +b_1 z+b_0)$ modulo $f$
is $2^n$. 
Hence, the number of tuples $\ovv$ that can be reached 
from $\ovv^\ast$ is not greater than $2^n\cdot 2^n = 4^n$.
\end{proof}

\section{Divisibility problem for parametric polynomials is decidable}
\label{se:divisibility-decidable}

Here we prove the main result of the paper.

\subsection{Piecewise periodicity of polynomials}

Let $p\in\MN$.
We say that a bit string $w=w_1\dots w_k$ is \emph{$p$-periodic}
if $w_{i}=w_{i+p}$ for every $1\le i\le k-p$. 
We do not require $p$ to be the minimal value satisfying this property. 
For a nonempty bit string $w$ and $k\in\MN$ define the bit string
$$
w^{(k)} = \mbox{the initial segment of $w^k$ of length $k$}.
$$

Consider a sequence of polynomials
$\{f_\delta\}_{\delta=1}^\infty \in\MZ_2[z^\pm]$.
Informally, we say that $\{f_\delta\}$ is \emph{$(A,p)$-piecewise periodic},
or, simply \emph{$(A,p)$-periodic}, if 
for all sufficiently large indices $\delta$, the binary presentation
for $f_\delta$ consists of $p$-periodic segments separated by short 
non-periodic segments of length $2A+1$.
Formally, $\{f_\delta\}$ is \emph{$(A,p)$-periodic} if 
\begin{itemize}
\item 
there are $s,t\in\MZ$, where $s\le t$,
\item 
there are bit strings $n_s,\dots,n_t\in\{0,1\}^{2A+1}$,
\item 
there are bit strings $p_{s+1},\dots,p_t\in\{0,1\}^{p}$,
\item 
there is $\Delta\in\MN$,
\end{itemize}
such that for every $i\ge \Delta$ the following condition holds:
\begin{itemize}
\item 
$\pres(f_\delta) = 
(\delta t+A,
n_s \circ p_{s}^{(\delta-2A-1)}
\circ \dots \circ 
n_{t-1} \circ p_{t-1}^{(\delta-2A-1)}
\circ
n_t
 )$.
\end{itemize}
Note that the bit strings $n_s,\dots,n_t\in\{0,1\}^{2A+1}$ and
$p_{s+1},\dots,p_t\in\{0,1\}^{p}$
correctly define
all but finitely many of the polynomials $f_\delta$;
concatenation in the second component of $\pres(f_\delta)$ 
is defined as expected when $\delta\ge 2A+1$.
That defines a subsequence of polynomials
$\{f_\delta\}_{\delta=2A+1}^\infty$,
denoted by $\PWP(t;n_s,\dots,n_t;p_{s+1},\dots,p_t)$, where
$\PWP$ stands for \emph{piecewise periodic}.
The tuple $\CP=(t;n_s,\dots,n_t;p_{s+1},\dots,p_t)$ is a representation
for the sequence $\{f_\delta\}_{\delta=2A+1}^\infty$.
Note that the values of $A$ and $p$ in that representation are defined 
by the lengths of bit strings $n_i$ and $p_i$. 

The trivial $\delta$-parametric polynomial $\blf=\blo$
defines a $(1,1)$-periodic sequence of polynomials.
The next lemma follows from the definition of $(A,p)$-periodicity.

\begin{lemma}\label{le:parametric-pwp}
Let $\blf=\sum_{i=s}^t f_i(z) z^{i\delta}$ be a nontrivial 
$\delta$-parametric polynomial, where 
$f_i(z)=\sum_{j=-\infty}^\infty f_{ij}z^j$.
Then the sequence $\{\blf_\delta\}$ is $(A,1)$-periodic, where
$$
A=\max_{f_{ij}\ne 0}|j|.
$$
\end{lemma}

Figure \ref{fi:pwp-example} illustrates Lemma \ref{le:parametric-pwp}
for a particular $\delta$-parametric polynomial.
\begin{figure}[h]
\centering
\includegraphics[width=0.9\linewidth]{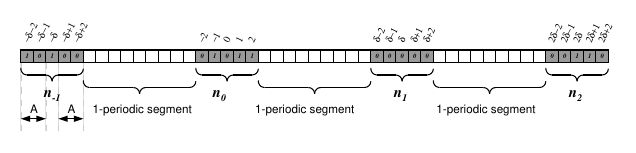}
\caption{
The $\delta$-parametric polynomial 
$
\blg=
(1+z^{-2})z^{-\delta}+
(z^2+z+z^{-1}) + 
(z+1)z^{2\delta}
$
defines an $(A=2,p=1)$-pwp sequence of polynomials, one of which, 
for $\delta=14$, is visualized above.
Increasing $\delta$ by $1$, stretches the picture by making 
the periodic segments' length greater by $1$.}
\label{fi:pwp-example}
\end{figure}

For a presentation $\CP=(t;n_s,\dots,n_t;p_{s+1},\dots,p_t)$
of a piecewise periodic sequence of functions 
define its \emph{$\delta$-degree} and \emph{$\delta$-order} as
$$
\deg_\delta(\CP) = t\ \ \mbox{ and }\ \ \ord_\delta(\CP) = s.
$$

\subsection{Reduction theorems for divisibility problems}

By Lemma \ref{le:parametric-pwp}
every $\delta$-parametric polynomial defines 
a sequence of piecewise periodic polynomials.
Hence, divisibility problem for $\delta$-parametric polynomials
can be viewed as a special case
(in sense of Proposition \ref{pr:top-div})
of the following problem.

\begin{algproblem}
\problemtitle{\textsc{Divisibility for piecewise periodic polynomials} $(\DIV(\blf,\CG,A,B,p,r))$.}
\probleminput{
$\blf$ is a $\delta$-parametric polynomial,
$\CG$ is a presentation of a $(B,p)$-periodic sequence of polynomials
$\{g_\delta\}_{\delta=2B+1}^\infty$, 
and $r\in\MN$.
}
\problemquestion{Is there $\delta\ge 2B+1$ satisfying $\delta\equiv_p r$ and $\blf_\delta\mid g_\delta$?}
\end{algproblem}

It is convenient to view $\DIV(\blf,\CG,A,B,p,r)$ as a true/false
value depending on whether $(\blf,\CG,A,B,p,r)$ is a positive/negative
instance of the problem.
Similarly, $\blf_\delta \mid g_\delta$ defines
a true/false value depending on whether
$\blf_\delta$ divides $g_\delta$ or not.

Let us introduce additional notation for divisibility.
\begin{itemize}
\item 
By $(\DIVw(\blf,\CG,A,B,p,r))$ we denote the set of all witnesses 
$\delta$ for the instance $(\blf,\CG,A,B,p,r)$ of the divisibility problem.
\item 
Similarly, we define $\DIV^{wit}_+(\blf,\blg)$.
\item 
For instantiated polynomials $\blf_\delta$ and $g_\delta$
the symbol $\blf_\delta \mid^{wit} g_\delta$ defines $\varnothing$ 
when $\blf_\delta \nmid g_\delta$ and $\{\delta\}$
when $\blf_\delta \mid g_\delta$.
\end{itemize}
The next proposition reduces an instance of the
divisibility problem $\DIV_+$
for $\delta$-parametric polynomials
to an instance of divisibility problem for piecewise periodic polynomials.

\begin{proposition}[The top case for divisibility]\label{pr:top-div}
Let $\blf$ be a nontrivial $\delta$-parametric polynomial and
$\{\blf_\delta\}_{\delta=0}^\infty$
an $(A,1)$-periodic sequence of polynomials.
Let $\blg$ be a $\delta$-parametric polynomial and
$\{\blg_\delta\}_{\delta=0}^\infty$
a $(B,1)$-periodic sequence of polynomials.
Let $\CG$ be a presentation for $\{\blg_\delta\}_{\delta=2B+1}^\infty$.
Then
$$
\DIV^{wit}_+(\blf,\blg) =
\bigg(\bigcup_{\delta=0}^{2B}
\blf_\delta \mid^{wit} \blg_\delta
\bigg)
\ \cup\ 
\DIV^{wit}(\blf,\CG,A,B,1,0).
$$
\end{proposition}

\begin{proof}
The statement follows from the definitions of 
$\DIV_+(\blf,\blg)$ and $\DIV(\blf,\CG,A,B,1,0)$.
\end{proof}

For a $\delta$-parametric polynomial $\blf$
and a representation $\CG$ of a $(B,p)$-periodic sequence of polynomials define
\begin{equation}\label{eq:Delta_f_CG}
\Delta(\blf,\CG) = [\deg_\delta(\CG) - \ord_\delta(\CG)]-
[\deg_\delta(\blf) - \ord_\delta(\blf)].
\end{equation}

Notice that a $(B,p)$-periodic sequence can be viewed as a
$(B',p)$-periodic, where $B<B'$. In particular, 
in Proposition \ref{pr:top-div} we may always assume that
$A\le B$.
The next proposition reduces an instance of the divisibility problem
with a presentation
$\CG$ to several instances of the same problem 
with presentations $\CG_i$ satisfying
$$
\ord_\delta(\CG_i)=\ord_\delta(\CG)
\ \ \mbox{ and }\ \ 
\deg_\delta(\CG_i) = \deg_\delta(\CG)-1,
$$
simplifying the problem.

\begin{proposition}[The reduction for divisibility]\label{pr:mid-div}
Let $\blf=\sum_{i=s}^t f_i(z) z^{i\delta}$ be a nontrivial 
$\delta$-parametric polynomial and $\{\blf_\delta\}_{\delta=0}^\infty$
an $(A,1)$-periodic sequence.
Let $\CG$ be a presentation
of a $(B,p)$-periodic sequence of polynomials 
$\{g_\delta\}_{\delta=2B+1}^\infty$, where $A\le B$.
Let $P=P_{f_t}$ be the constant associated with the polynomial $f_t$
introduced in Proposition \ref{pr:periodic-division}. 
If $\Delta\ge 0$, then for any $0\le r\le p-1$
$$
\DIV^{wit}(\blf,\CG,A,B,p,r) 
= 
\bigg(\bigcup_{\delta=2B+1,\delta\equiv_p r}^{2B+8A}
\blf_\delta \mid^{wit} g_\delta
\bigg)
\ \cup\ 
\bigg(\bigcup_{i=0}^{P-1} \DIV^{wit}(\blf,\CG_i,A,B+4A,p\cdot P,pi+r)\bigg)
$$
for some effectively computable $(B+4A,p\cdot P)$-periodic
sequences of polynomials defined by presentations
$\CG_0,\dots\CG_{P-1}$.
\end{proposition}

\begin{proof}
By definition,
\begin{align*}
\DIV^{wit}(\blf,\CG,A,B,p,r) 
&= \Set{\delta}{
\delta\ge 2B+1,\ \ 
\delta\equiv_p r,\ \ 
\blf_\delta\mid g_\delta}\\
&= 
\bigg(\bigcup_{\delta=2B+1,\delta\equiv_p r}^{2B+8A}
\blf_\delta \mid^{wit} g_\delta
\bigg)
\cup
\Set{\delta}{
\delta\ge 2B+8A+1,\ \ 
\delta\equiv_p r,\ \ 
\blf_\delta\mid g_\delta}.
\end{align*}
Below we prove that for properly chosen presentations $\CG_0,\dots,\CG_{P-1}$ we have
\begin{equation}\label{eq:14}
\Set{\delta}{
\delta\ge 2B+8A+1,\ \ 
\delta\equiv_p r,\ \ 
\blf_\delta\mid g_\delta}=
\bigcup_{i=0}^{P-1} \DIV^{wit}(\blf,\CG_i,A,B+4A,p\cdot P,pi+r).
\end{equation}
For that purpose we introduce a notion of \emph{partial division}
of piecewise periodic polynomials $\{g_\delta\}$
by polynomials $\{\blf_\delta\}$.

Polynomial division of $g$ by $f$ can be viewed as a process 
of elimination of the leading monomial in $g$ 
by a suitable multiple of $f$.
This process is repeated through several steps until the remainder is obtained.

\begin{figure}[h]
\centering
\includegraphics[scale=1.2]{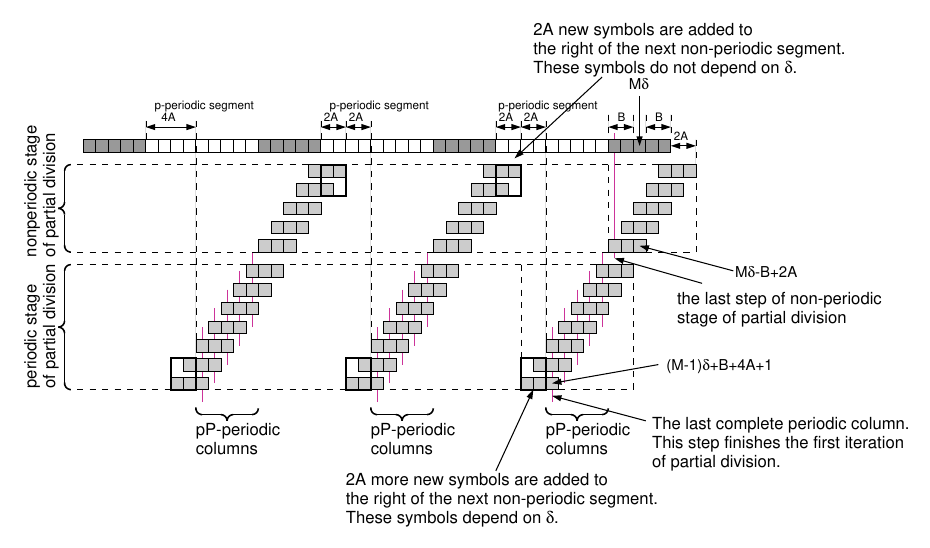}
\caption{A schematic picture for a single iteration of partial division.}
\label{fi:partial_division}
\end{figure}

For piecewise periodic polynomials $\{g_\delta\}$
and $\{\blf_\delta\}$
the partial division of $g_\delta$ by $\blf_\delta$
eliminates monomials in $g_\delta$ in the same way,
but stops before the complete remainder is found.
A schematic picture for a single iteration of partial division
is shown in Figure \ref{fi:partial_division}.
Gray squares stand for monomials (their potential positions) 
in non-periodic segments in $\blf_\delta$ and $g_\delta$
(originally they are defined by $f_i$'s and $g_i$'s).
White squares stand for the monomials in $p$-periodic segments
in $\blf_\delta$ and $g_\delta$.
Let $M=\deg_\delta(\CG)$.
We distinguish two stages in the process of partial division.
\begin{itemize}
\item 
Elimination of the initial non-periodic segment in $g_\delta$
before reaching the subsequent periodic segment.
The monomial $z^{M\delta-B+2A}$ is the last
monomial eliminated in $g_\delta$ at this stage.
\item 
Elimination of the subsequent periodic segment before reaching 
the next (extended) non-periodic segment.
The monomial $z^{(M-1)\delta+B+2A+1}$ is the last
monomial eliminated in $g_\delta$ at this stage.
\end{itemize}
Partial division begins by eliminating the leading monomial in $\blg_\delta$.
For that we consider the rightmost overlap of leading terms
as shown in Figure \ref{fi:partial_division}.
A single iteration of partial division eliminates the monomials
in the range
$z^{M\delta+B},\dots,z^{(M-1)\delta+B+4A+1}$ from $g_\delta$.
Eliminating those monomials produces the \emph{remainder of partial division}
$r_\delta$ that depends on $\delta$.
We claim that the sequence $\{r_\delta\}$ is in some sense piecewise periodic
(more precisely it splits into $P_f$ piecewise periodic sequences
that have their own non-periodic and periodic segments).

The elimination is primarily done by the term $f=f_t$
(it contains the leading monomial of $\blf$)
and the sequence of copies of $\blf_\delta$ in 
Figure \ref{fi:partial_division} is uniquely defined by
the transition chain in $\Gamma_f$ that starts at $0$
and is labeled with the sequence of monomial coefficients from $g_\delta$.
Suppose that the transition chain in $\Gamma_f$ that starts at $0$
and is labeled with the coefficients for monomials 
$z^{M\delta+B},\dots,z^{M\delta-B}$ ends at some state $s_0$.
$\CG$ is a presentation of a $(B,p)$-periodic sequence of polynomials
and, hence, the sequence of coefficients 
for monomials $z^{M\delta-B-1},\dots,z^{(M-1)\delta+B+4A+1}$
in $f_\delta$ is $p$-periodic.
If follows from Proposition \ref{pr:periodic-division} that
elimination of those monomials adds a $pP_f$-periodic 
sequence of copies of $f_\delta$'s to Figure \ref{fi:partial_division}
(horizontal sequences of squares corresponding to $\blf_\delta$)
at the periodic stage of partial division.
Hence, the columns of squares for the monomials 
$z^{M\delta-B-1},\dots,z^{(M-1)\delta+B+4A+1}$ are $pP_f$-periodic too.
Consequently, the sums over those columns are $pP_f$-periodic as well.
That creates $pP_f$-periodic segments in the reminder $r_\delta$.

As for non-periodic segments,
as shown in Figure \ref{fi:partial_division},
each original non-periodic segment 
(except for the first one which is eliminated) 
is extended with $4A$ additional symbols on the right.
Non-periodic segments do not extend left, because
periodic columns start immediately after every non-periodic part.
Therefore, the $B$ value for the new non-periodic segments
increases by $4A$.

Therefore, increasing $\delta$ by $pP_f$ adds a single period to each periodic 
segment (of length $pP_f$) in the remainder of partial division
$r_\delta$ and does not change non-periodic segments
(with the $B$ value increased by $4A$). 
Now, congruence modulo $p$ in the LHS of \eqref{eq:14}
can be split into $P$ congruences modulo $pP$
\begin{align*}
\Set{\delta}{
\delta\ge 2B+8A+1,\ \ 
\delta\equiv_p r,\ \ 
\blf_\delta\mid g_\delta}
&=
\bigcup_{i=1}^P
\Set{\delta}{
\delta\ge 2B+8A+1,\ \ 
\delta\equiv_{pP} pi+r,\ \ 
\blf_\delta\mid g_\delta}\\
&=
\bigcup_{i=1}^P
\Set{\delta}{
\delta\ge 2B+8A+1,\ \ 
\delta\equiv_{pP} pi+r,\ \ 
\blf_\delta\mid r_\delta}.
\end{align*}
Since we consider $\delta$ modulo $pP$,
for every $i$ the sequence of polynomials 
$\{r_\delta\}_{\delta\equiv_{pP} pi+r}$ is a subsequence
of a single $(B+4A,pP)$-periodic sequence of polynomials
(as we explained, 
those polynomials have the same non-periodic and periodic segments).
Hence, every $i$ we have
$$
\Set{\delta}{
\delta\ge 2B+8A+1,\ \ 
\delta\equiv_{pP} pi+r,\ \ 
\blf_\delta\mid r_\delta}=\DIV^{wit}(\blf,\CG_j,A,B+4A,p\cdot P,pj+r).
$$
Finally, $\CG_i$ represents polynomials $r_{pi+r}$;
it can be computed directly by (partially) dividing $g_\delta$ by $\blf_\delta$
and extracting periodic and non-periodic segments.
\end{proof}

\begin{proposition}[The base case for divisibility]\label{pr:bottom-div}
Let $\blf=\sum_{i=s}^t f_i(z) z^{i\delta}$ be a nontrivial 
$\delta$-periodic polynomial and $\{\blf_\delta\}_{\delta=0}^\infty$
an $(A,1)$-periodic sequence.
Let $\CG$ be a presentation
of a $(B,p)$-periodic sequence of polynomials 
$\{g_\delta\}_{\delta=2B+1}^\infty$.
If $(\blf,\CG,A,B,p,r)$ is a positive instance of the divisibility problem, 
then it has a witness $\delta$ satisfying the following:
\begin{enumerate}
\item[(1)]
$
\Delta(\blf,\CG) = 0
\ \ \Rightarrow\ \ \delta\le p+2B+4A+1
$,
\item[(2)]
$
\Delta(\blf,\CG) < 0
\ \ \Rightarrow\ \ \delta \le 2B + 2A
$.
\end{enumerate}
\end{proposition}

\begin{proof}
Let us consider several cases.

\textsc{(Case-I)}
Suppose that $\Delta(\blf,\CG) = 0$.
This case is visualized in Figure \ref{fi:division-same-span}.
The actual steps of division are defined by the leading and trailing 
monomials in $\blf_\delta$ and $g_\delta$ for $\delta \ge 2B+1$
(the latter condition comes from the definition of $\DIV$)
and the number of steps is bounded by $2(A+B)+1$ for any $\delta$
as shown in the figure.
\begin{figure}[h]
\centering
\includegraphics[width=1\linewidth]{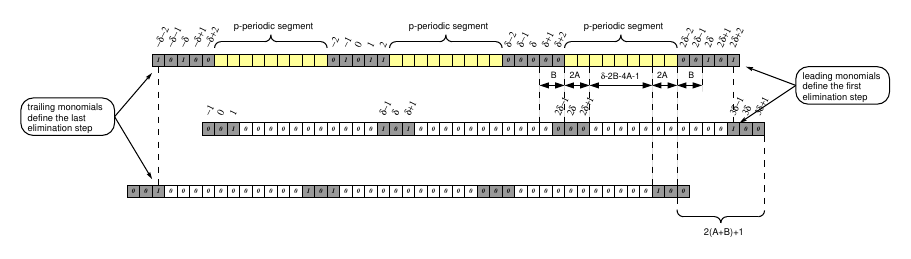}
\caption{A schematic picture for division $\blf_\delta\mid g_\delta$
for the case
$\deg_\delta(\CG) - \ord_\delta(\CG) = \deg_\delta(\blf) - \ord_\delta(\blf)=3$.}
\label{fi:division-same-span}
\end{figure}
Consider two subcases.

\textsc{(Case-Ia)}
Suppose that some $p$-periodic segment $p_i$
in the representation $\CG$ is not a sequence of $p$ zeros;
that is, $p_i$ contains at least one symbol $1$.
Then $\delta$ cannot be large, because $2(A+B)+1$ steps
of division modify only a bounded part of the corresponding 
$p$-periodic segment,
leaving a subsegment of length at least $\delta - 2B - 4A - 1$ 
unmodified (see Figure \ref{fi:division-same-span}).
Therefore, every witness $\delta$ must satisfy $\delta - 2B - 4A - 1 < p$,
because otherwise a whole period survives the division and
at least one symbol $1$ appears in the remainder.

\textsc{(Case-Ib)}
Suppose that every periodic bit string $p_i$
in the representation $\CG$ is a sequence of $p$ zeros, i.e., $p_i = 0^p$.
If $\delta > p + 2B + 4A + 1$ is a witness, then $\delta' = \delta - p$
is also a witness, because decreasing $\delta$ by $p$ simply removes
one full period (consisting of $0^p$) from the middle of each $p$-periodic segment.
As a result, the remainder, which is a sequence of zeros, shortens,
but remains a sequence of zeros.
Therefore, in this case, if there is a large solution, then there is a solution
satisfying $\delta\le p+2B+4A+1$.

\textsc{(Case-II)}
Finally, consider the case when 
$\deg_\delta(\CG) - \ord_\delta(\CG) < \deg_\delta(\blf) - \ord_\delta(\blf)$.
If $\delta\in\MN$ is a witness for the given instance,
then $\blf_\delta \mid g_\delta$ and, hence,

\begin{center}
\scalebox{0.9}{\vbox{
\begin{align*}
(\deg_\delta(\blf) - \ord_\delta(\blf))\delta - 2A
&\ \le\ 
\deg(\blf_\delta)-\ord(\blf_\delta)
&&(\mbox{because $\{\blf_\delta\}$ is $(A,1)$-periodic} )\\
&\ \le\ 
\deg(g_\delta)-\ord(g_\delta)
&& (\mbox{because} \blf_\delta \mid g_\delta)\\
&\ \le\ 
(\deg_\delta(\CG) - \ord_\delta(\CG))\delta +2B
&& (\mbox{because $\{g_\delta\}$ is $(B,p)$-periodic.})
\end{align*}
}}
\end{center}
which implies that
$$
[(\deg_\delta(\blf) - \ord_\delta(\blf)) -
(\deg_\delta(\CG) - \ord_\delta(\CG))]\delta  
\le 2B + 2A.
$$
Hence, $\delta \le 2B + 2A$ and (2) holds.
\end{proof}

\subsection{Divisibility for $\delta$-parametric polynomials is decidable}

Consider an instance $(\blf,\blg)$ of $\DIV_+$, where $\blf$ is not trivial.
If $\blg$ is trivial, then the instance is positive. 
Hence, we assume that $\blg$ is not trivial too.
Suppose that $\{\blf_\delta\}_{\delta=0}^\infty$
is an $(A,1)$-periodic sequence of polynomials and
$\{\blg_\delta\}_{\delta=0}^\infty$ a $(B,1)$-periodic sequence 
of polynomials satisfying $A\le B$.
Define
\begin{equation}\label{eq:Delta_fg}
\Delta(\blf,\blg) = [\deg_\delta(\blg) - \ord_\delta(\blg)]-
[\deg_\delta(\blf) - \ord_\delta(\blf)].
\end{equation}

Propositions \ref{pr:top-div}, \ref{pr:mid-div}, and \ref{pr:bottom-div}
give an algorithm for computing $\DIV_+(\blf,\blg)$ for given 
$\delta$-parametric polynomials 
$\blf$ and $\blg$.

\begin{theorem}\label{th:div-plus-bound}
Let $\blf=\sum_{i=s}^t f_i(z) z^{i\delta}$
be a nontrivial $\delta$-parametric polynomial and
$\{\blf_\delta\}_{\delta=0}^\infty$
an $(A,1)$-periodic sequence of polynomials.
Let $\blg$ be a $\delta$-parametric polynomial and
$\{\blg_\delta\}_{\delta=0}^\infty$
a $(B,1)$-periodic sequence of polynomials satisfying $A\le B$.
Let $\Delta=\Delta(\blf,\blg)$ and $P=P_{f_t}$.
Then the following holds:
\begin{itemize}
\item[(a)]
$
\Delta<0\ \ \Rightarrow\ \ 
\DIV_+(\blf,\blg) = \bigvee_{\delta=0}^{2B+2A} (\blf_\delta \mid \blg_\delta),
$
\item[(b)]
$
\Delta\ge 0\ \ \Rightarrow\ \ 
\DIV_+(\blf,\blg) = \bigvee_{\delta=0}^{P^\Delta+2(B+4A\Delta)+4A+1} (\blf_\delta \mid \blg_\delta).
$
\end{itemize}
\end{theorem}

\begin{proof}
In both cases, 
Proposition \ref{pr:top-div} constructs a pwp-presentation $\CG$ for $\blg$
and translates $\DIV_+(\blf,\blg)$ into 
$\DIV(\blf,\CG,A,B,1,0)$ with an overhead of 
computing the disjunction
$\bigvee_{\delta=0}^{2B} (\blf_\delta \mid \blg_\delta)$.
Notice that $\Delta(\blf,\CG)=\Delta(\blf,\blg)$.

To prove (a) suppose that $\Delta<0$.
Then the assumption of Proposition \ref{pr:bottom-div}(2) 
is satisfied for $\DIV(\blf,\CG,A,B,1,0)$.
Hence, a witness $\delta$ for $\DIV(\blf,\CG,A,B,1,0)$ (if exists)
is bounded by $2A+2B$. Overall that translates to 
$\bigvee_{\delta=0}^{2B+2A} (\blf_\delta \mid \blg_\delta)$.

To prove (b) suppose that $\Delta\ge 0$.
Proposition \ref{pr:mid-div} translate the instance
$\DIV(\blf,\CG,A,B,1,0)$ produced by Proposition \ref{pr:top-div}
into a disjunction of $P$ instances $\DIV(\blf,\CG_i,A,B+4A,1\cdot P,i)$.
Each new instance of $\DIV$ can be further split
to a disjunction of $P$ similar instances. This creates a tree
$T_{\blf,\blg}$ of instances $\DIV(\blf,\CG,A,B',p,r)$ with 
the root $\DIV(\blf,\CG,A,B,1,0)$.
Each reduction decreases the value of $\Delta$ by one.
Reduction stops when the value of $\Delta$ reaches zero,
which creates a tree of height $\Delta$,
where every non-leaf vertex has degree $P$,
and every leaf node has instance 
$\DIV(\blf,\CG_i,A,B+4A\Delta,P^\Delta,i)$ for which $\Delta=0$.
By Proposition \ref{pr:mid-div}, 
$\DIV(\blf,\CG,A,B,1,0)$ has a witness $\delta$ if and only if 
$\delta$ is a small instance (less than $2(B+4A\Delta)+8A$) 
of $\DIV(\blf,\CG,A,B,1,0)$
or $\delta$ is a witness for one of the instances
$\DIV(\blf,\CG_i,A,B+4A\Delta,P^\Delta,i)$.
By Proposition \ref{pr:bottom-div}(1), in the latter case 
we may assume that that $\delta$ is bounded by $P^\Delta+2(B+4A\Delta)+4A+1$.

Thus, in this case, a witness $\delta$ for $\DIV(\blf,\CG,A,B,1,0)$ (if exists)
is bounded by $P^\Delta+2(B+4A\Delta)+4A+1$. Overall that translates to 
$\bigvee_{\delta=0}^{P^\Delta+2(B+4A\Delta)+4A+1} (\blf_\delta \mid \blg_\delta)$.
\end{proof}

\begin{corollary}\label{co:div-bound}
Let $\blf=\sum_{i=s}^t f_i(z) z^{i\delta}$
be a nontrivial $\delta$-parametric polynomial and
$\{\blf_\delta\}_{\delta=0}^\infty$
an $(A,1)$-periodic sequence of polynomials.
Let $\blg$ be a $\delta$-parametric polynomial and
$\{\blg_\delta\}_{\delta=0}^\infty$
a $(B,1)$-periodic sequence of polynomials satisfying $A\le B$.
If $(\blf,\blg)$ is a positive instance of $\DIV$, then it has a 
witness $\delta$ satisfying 
$$
-(P_{f_s}^\Delta+2(B+4A\Delta)+4A+1)
\le \delta \le 
P_{f_t}^\Delta+2(B+4A\Delta)+4A+1.
$$
\end{corollary}

\begin{proof}
By Proposition \ref{pr:div-div2},
$\DIV(\blf,\blg) = \DIV_+(\blf,\blg) \vee \DIV_+(\flip(\blf),\flip(\blg))$.
Hence, $(\blf,\blg)$ is a positive instance of $\DIV$ if and only if
either $(\blf,\blg)$ is a positive instance of $\DIV_+$
or $(\flip(\blf),\flip(\blg))$ is a positive instance of $\DIV_+$.
Theorem \ref{th:div-plus-bound} 
covers the instance $\DIV_+(\blf,\blg)$ and 
gives the upper bound on $\delta$.

Notice that applying $\flip$
to $\blf$ and $\blg$ does not change the values of $A$, $B$, and $\Delta$,
but changes the leading block in $\blf$.
The leading $\delta$-block of $\flip(\blf)$ is $f_s(z)$.
Hence, if $(\flip(\blf),\flip(\blg)$ is a positive instance of 
$\DIV_+$, then it has a witness $\delta'\ge 0$
bounded by $P_{f_s}^\Delta+2(B+4A\Delta)+4A+1$.
Hence, $-\delta'$ is a witness for the instance
$(\blf,\blg)$ of $\DIV$ that satisfies the claimed lower bound.
\end{proof}

\begin{theorem}\label{th:delta-bound}
Suppose that $w=w(a,x,t)$ satisfies $x_w=0$.
If $w=1$ has a solution, then it has a solution $x=(\delta,f)$ satisfying
$|\delta| \le 2^{|w|^2/2}+2|w|^2+3|w|+1$.
\end{theorem}

\begin{proof}
To decide if $w=1$ has a solution the algorithm
computes $\blg=\num(w)$ and $\blf=\den(w)=\sum_{i=s}^t f_i(z)z^{i\delta}$. 
Since $x_w=0$, it follows that $t_w=0$; otherwise the equation $w=1$ has no solution. 
Then the following holds:
\begin{itemize}
\item 
$\deg_\delta(\blf)-\ord_\delta(\blf) \le \tfrac{1}{2}|w|$,
\item 
$\deg_\delta(\blg)-\ord_\delta(\blg) \le \tfrac{1}{2}|w|$,
\item 
$A,B,\Delta \le \tfrac{1}{2}|w|$,
\item 
$\deg(f_i) - \ord(f_i) \le \tfrac{1}{2}|w|$ for every $\delta$-block $f_i$ in $\blf$,
\end{itemize}
because the sets $N_w,D_w\subset \MZ^2$ 
can be enclosed in a $\tfrac{1}{2}|w| \times \tfrac{1}{2}|w|$ 
square; otherwise the trace of $w$ does not return to $(0,0)$.
Hence, by Lemma \ref{le:P-bound},
$P_{f_i} \le 4^{|w|/2} = 2^{|w|}$
for every non-trivial $\delta$-block $f_i$ in $\blf$.

Finally, the equation $w=1$ has a solution if and only if $(\blf,\blg)$
is a positive instance of $\DIV$.
By Corollary \ref{co:div-bound}, the latter is true if and only if
$(\blf,\blg)$ has a witness $\delta$ bounded by
\begin{align*}
|\delta| 
& \le (\max P_{f_i})^\Delta+2(B+4A\Delta)+4A+1 \\
& \le 2^{|w|\Delta}+2(|w|/2+|w|^2)+2|w|+1 \\
& \le 2^{|w|^2/2}+2|w|^2+3|w|+1.
\end{align*}
\end{proof}

We do not provide a time complexity estimate for the algorithm in the next statement,
since it is clearly super-exponential.

\begin{corollary}
There exists an algorithm that, given a one-variable equation $w(a,t,x)=1$
satisfying $\sigma_x(w)=0$, decides whether the equation has a solution or not.
\end{corollary}

\begin{proof}
For a given equation $w=1$ compute $\num(w)$ and $\den(w)$.
Enumerate the values of $\delta$ in the range 
$$
-(2^{|w|^2/2}+2|w|^2+3|w|+1)\le \delta \le 2^{|w|^2/2}+2|w|^2+3|w|+1.
$$
For each particular value of $\delta$ compute $\num_\delta(w)$ and $\den_\delta(w)$ and
notice that 
\begin{align*}
\deg(\num_\delta) - \ord(\num_\delta) &\le |w|^2/2 + |w|^2/2(2^{|w|^2/2}+2|w|^2+3|w|+1),\\
\deg(\den_\delta) - \ord(\den_\delta) &\le |w|^2/2 + |w|^2/2(2^{|w|^2/2}+2|w|^2+3|w|+1).
\end{align*}
Check if $\den_\delta \mid \num_\delta$ for some $\delta$ and if so, return Yes.
If $\den_\delta \nmid \num_\delta$ for every $\delta$, then return No.
\end{proof}

\bibliography{main_bibliography}

@article{Appel:1968,
  author  = {{Appel}, K.},
  title   = {One-Variable Equations in Free Groups},
  journal = {Proceedings of the American Mathematical Society},
  volume  = {19},
  year    = {1968},
  pages   = {912--918},
}

@ARTICLE{AS,
   author = {{Appel}, K. and {Schupp}, P.},
    title = "{Artin groups and infinite Coxeter groups}",
   JOURNAL = {Invent. Math.},
  FJOURNAL = {Inventiones Mathematicae},
    VOLUME = {72},
      YEAR = {1983},
     PAGES = {201--220},
}

@article{Bartholdi-Dong-Pernak-Waechter:2024,
  title        = {Equations in wreath products},
  author       = {{Bartholdi}, L. and {Dong}, R. and {Pernak}, L. and {Wächter}, J. P.},
  journal      = {arXiv preprint arXiv:2410.04905},
  year         = {2024},
  url          = {https://arxiv.org/abs/2410.04905},
  eprint       = {2410.04905},
}

@article{Baumslag:1961,
   author = {{Baumslag}, G.},
    title = "{Wreath products and finitely presented groups}",
  journal = {Math. Z.},
 fjournal = {Mathematische Zeitschrift},
   volume = {75},
    pages = {22--28},
     year = {1961},
}

@ARTICLE{Bormotov-Gilman-Miasnikov:2008,
   author = {{Bormotov}, D. and {Gilman}, R. and {Miasnikov}, A.},
    title = "{Solving one-variable equations in free groups}",
   JOURNAL = {J. Group Theory},
  FJOURNAL = {Journal of Group Theory},
    VOLUME = {12},
      YEAR = {2008},
     PAGES = {317--330},
}

@inproceedings{Dong:2025,
   author = {{Dong}, R.},
    title = "{Linear equations with monomial constraints and decision problems in
abelian-by-cyclic groups}",
 BOOKTITLE = {Proceedings of the 2025 Annual ACM-SIAM Symposium on Discrete Algorithms (SODA)},
      YEAR = {2025},
     PAGES = {1892--1908},
}

@InProceedings{Ferens-Jez:2021,
  author =	{{Ferens}, R. and {Je\.{z}}, A.},
  title =	{{Solving One Variable Word Equations in the Free Group in Cubic Time}},
  booktitle =	{38th International Symposium on Theoretical Aspects of Computer Science (STACS 2021)},
  pages =	{30:1--30:17},
  series =	{Leibniz International Proceedings in Informatics (LIPIcs)},
  year =	{2021},
  volume =	{187},
  editor =	{Bl\"{a}ser, Markus and Monmege, Benjamin},
  publisher =	{Schloss Dagstuhl -- Leibniz-Zentrum f{\"u}r Informatik},
  address =	{Dagstuhl, Germany},
}

@incollection{Gilman-Myasnikov:2004,
  author    = {{Gilman}, R. and {Myasnikov}, A.},
  title     = {One variable equations in free groups via context free languages},
  booktitle = {Groups, Languages, Algorithms},
  series    = {Contemporary Mathematics},
  volume    = {349},
  pages     = {83--88},
  year      = {2004},
  publisher = {American Mathematical Society}
}

@BOOK{G,
   author = {{Goldreich}, O.},
    title = "{Foundations of Cryptography: Volume 1, Basic Tools}",
 PUBLISHER = {Cambridge University Press},
      YEAR = {2001},
}

@article{Jones:2004,
  author       = {{Jones}, G.},
  title        = "{On the Markov chain central limit theorem}",
  journal      = {Probab. Surv.},
 fjournal      = {Probability Surveys},
  volume       = {1},
  pages        = {299--320},
  year         = {2004},
  doi          = {10.1214/154957804100000051},
  url          = {https://doi.org/10.1214/154957804100000051}
}

@article{Kapovich_Miasnikov:2002,
    AUTHOR = {{Kapovich}, I. and {Miasnikov}, A.~G.},
     TITLE = {Stallings foldings and subgroups of free groups},
   JOURNAL = {J. Algebra},
  FJOURNAL = {Journal of Algebra},
    VOLUME = {248},
      YEAR = {2002},
     PAGES = {608--668},
}

@ARTICLE{Kharlampovich-Lopez-Miasnikov:2020,
    AUTHOR = {{Kharlampovich}, O. and {Lopez}, L. and {Miasnikov}, A.},
    TITLE = "{Diophantine Problem in Some Metabelian Groups}",
    JOURNAL = {Mathematics of Computation},
    VOLUME = {89},
    YEAR = {2020},
    ISSUE = {325},
    PAGES = {2507},
}

@article{Lorents:1963,
  author  = {{Lorents}, A.},
  title   = {The solution of systems of equations in one unknown in free groups},
  journal = {Doklady Akademii Nauk SSSR},
  volume  = {148},
  year    = {1963},
  pages   = {1253--1256},
  note    = {In Russian; translated in \emph{Soviet Mathematics Doklady} \textbf{4} (1963), 1230--1233}
}

@article{Lorents:1968,
  author  = {{Lorents}, A.},
  title   = {Representations of sets of solutions of systems of equations with one unknown in a free group},
  journal = {Doklady Akademii Nauk SSSR},
  volume  = {178},
  year    = {1968},
  pages   = {290--292},
  note    = {In Russian; translated in \emph{Soviet Mathematics Doklady} \textbf{9} (1968), 390--393}
}

@ARTICLE{Lysenok-Ushakov:2015,
    AUTHOR = {{Lysenok}, I. and {Ushakov}, A.},
     TITLE = "{Spherical quadratic equations in free metabelian groups}",
     JOURNAL = {Proc. Amer. Math. Soc.},
  FJOURNAL = {Proceedings of the American Mathematical Society},
    VOLUME = {144},
      YEAR = {2016},
     PAGES = {1383--1390},
}

@article{Lysenok-Ushakov:2021,
title = {Orientable quadratic equations in free metabelian groups},
journal = {Journal of Algebra},
volume = {581},
pages = {303-326},
year = {2021},
issn = {0021-8693},
doi = {https://doi.org/10.1016/j.jalgebra.2021.04.013},
author = {{Lysenok}, I. and {Ushakov}, A.},
}

@ARTICLE{Mandel-Ushakov:2023b,
    AUTHOR = {{Mandel}, R. and {Ushakov}, A.},
     TITLE = {Quadratic equations in metabelian {B}aumslag-{S}olitar groups},
   JOURNAL = {Int. J. Algebra Comput.},
  FJOURNAL = {International Journal of Algebra and Computation},
    VOLUME = {33},
    NUMBER = {6},
        YEAR = {2023},
     PAGES = {1195--1216},
}

@article{Myasnikov-Romanovskii:2012,
  author    = {{Myasnikov}, A. and {Romanovskii}, N.},
  title     = {Universal theories for rigid soluble groups},
  journal   = {Algebra and Logic},
  volume    = {50},
  pages     = {539--552},
  year      = {2012},
  doi       = {10.1007/s10469-012-9182-1}
}

@book{MSU_book:2011,
    AUTHOR = {{Miasnikov}, A.~G. and {Shpilrain}, V. and {Ushakov}, A.},
     TITLE = {Non-Commutative Cryptography and Complexity of Group-Theoretic Problems},
    SERIES = {Mathematical Surveys and Monographs},
 PUBLISHER = {AMS},
      YEAR = {2011},
}

@inproceedings{P,
    AUTHOR = {{Petrides}, G.},
     TITLE = {Cryptanalysis of the public key cryptosystem based on the word problem on the Grigorchuk groups},
 BOOKTITLE = {9th IMA International Conference on Cryptography and Coding},
    SERIES = {Lecture Notes Comp. Sc.},
 PUBLISHER = {Springer},
    VOLUME = {2898},
      YEAR = {2003},
     PAGES = {234--244},
}

@article{Repin:1984,
  author    = {{Repin}, N.},
  title     = "{Solvability of equations with one indeterminate in nilpotent groups}",
  journal   = {Izvestiya Akademii Nauk SSSR. Seriya Matematicheskaya},
  volume    = {48},
  year      = {1984},
  number    = {6},
  pages     = {1229--1255},
  note      = {In Russian; English translation in \emph{Math. USSR-Izv.} \textbf{25} (1985), no. 3, 573--599}
}

@article{Ushakov-Weiers:2023,
title = {Quadratic equations in the lamplighter group},
author = {{Ushakov}, A. and {Weiers}, C.},
   JOURNAL = {J. Symbolic Comput.},
  FJOURNAL = {Journal of Symbolic Computation},
volume = {129},
pages = {102417},
year = {2025},
}

@MISC{Ushakov-Weiers:2025,
  author       = {{Ushakov}, A. and {Weiers}, C.},
  title        = {Orientable quadratic equations in wreath products of abelian groups},
  journal      = {arXiv preprint arXiv:2503.01929},
  year         = {2025},
  url          = {https://arxiv.org/abs/2503.01929},
  eprint       = {2503.01929},
  archivePrefix= {arXiv},
  primaryClass = {math.GR},
}

@article{Ushakov:2024,
  author       = {{Ushakov}, A.},
  title        = {Constrained inhomogeneous spherical equations: average-case hardness}, 
  journal      = {Groups Complex. Cryptol.},
 FJOURNAL      = {Journal of Groups, Complexity, Cryptology},
  volume       = {16},
  number       = {1, Special issue in memory of Ben Fine},
  year         = {2024},
  doi          = {10.46298/jgcc.2024.16.1.13555},
  url          = {https://gcc.episciences.org/13555},
}

@ARTICLE{W,
    AUTHOR = {{Weinbaum}, C.~M.},
     TITLE = "{On relators and diagrams for groups with one defining relator}",
   JOURNAL = {Illinois J.Math.},
  FJOURNAL = {Illinois Journal of Mathematics},
    VOLUME = {16},
      YEAR = {1972},
     PAGES = {308--322},
}

@article{Chiswell-Remeslennikov:2000,
  author    = {I. M. Chiswell and V. N. Remeslennikov},
  title     = {Equations in free groups with one variable. I},
  journal   = {Journal of Group Theory},
  volume    = {3},
  number    = {4},
  pages     = {445--466},
  year      = {2000},
}

@book{CoxLittleOSheaIVA,
  author    = {D. A. Cox and J. Little and D. O'Shea},
  title     = {Ideals, Varieties, and Algorithms: An Introduction to Computational Algebraic Geometry and Commutative Algebra},
  publisher = {Springer},
  year      = {2015},
  edition   = {4th},
  isbn      = {978-3-319-16720-6}
}

@book{GathenGerhard2003,
  author    = {J. von zur Gathen and J. Gerhard},
  title     = {Modern Computer Algebra},
  publisher = {Cambridge University Press},
  year      = {2003},
  edition   = {2},
  doi       = {10.1017/CBO9781139856065}
}

@article{Levine2022_virtually_class2_nilpotent,
  author       = {{Levine}, A.},
  title        = {Equations in virtually class 2 nilpotent groups},
  journal      = {Journal of Groups, Complexity, Cryptology},
  volume       = {14},
  number       = {1},
  pages        = {2:1--2:17},
  year         = {2022},
}

@article{DuchinLiangShapiro2015_equations_nilpotent_groups,
  author       = {Duchin, Moon and Liang, Hao and Shapiro, Michael},
  title        = {Equations in nilpotent groups},
  journal      = {Proceedings of the American Mathematical Society},
  volume       = {143},
  number       = {11},
  pages        = {4723--4731},
  year         = {2015},
  doi          = {10.1090/S0002-9939-2015-12365-4},
  eprint       = {arXiv:1401.2471},
}

\end{document}